\documentclass[10pt]{amsart}

\usepackage{amssymb}
\usepackage{palatino}
\usepackage{xcolor} 
\usepackage{graphicx}
\usepackage{multirow}

\usepackage{url}
\usepackage{cite}

\usepackage[margin=1in]{geometry}

\input xy
\xyoption{all}
\xyoption{2cell}
\UseComputerModernTips
\UseTwocells

\numberwithin{equation}{section}
\numberwithin{figure}{section}

\makeatletter
\newtheorem*{rep@theorem}{\rep@title}
\newcommand{\newreptheorem}[2]{
\newenvironment{rep#1}[1]{
\def\rep@title{#2 ##1}
\begin{rep@theorem}}
{\end{rep@theorem}}}
\makeatother

\newtheorem{theorem}{Theorem}[section]
\newreptheorem{theorem}{Theorem}
\newreptheorem{proposition}{Proposition}
\newreptheorem{corollary}{Corollary}
\newtheorem{lemma}[theorem]{Lemma}
\newtheorem{proposition}[theorem]{Proposition}
\newtheorem{corollary}[theorem]{Corollary}

\newtheorem{remark}[theorem]{Remark}

\theoremstyle{definition}
\newtheorem{definition}[theorem]{Definition}
\newtheorem{example}[theorem]{Example}

\newcommand{\C}{{\mathbb{C}}}

\newcommand{\Z}{{\mathbb{Z}}}

\newcommand{\R}{{\mathbb{R}}}

\newcommand{\T}{{\mathbb{T}}}
\renewcommand{\P}{{\mathbb{P}}}
\newcommand{\calP}{{\mathcal{P}}}

\newcommand{\KD}{{K_{\mathrm{D}}}}

\definecolor{blue-violet}{rgb}{0.54, 0.17, 0.89}
\definecolor{ao(english)}{rgb}{0.0, 0.5, 0.0}
\definecolor{awesome}{rgb}{1.0, 0.13, 0.32}

\DeclareMathOperator{\Stab}{Stab}

\DeclareMathOperator{\Hom}{Hom}

\DeclareMathOperator{\Tor}{Tor}

\DeclareMathOperator{\rank}{rank}

\DeclareMathOperator{\im}{im}

\DeclareMathOperator{\lcm}{lcm}

\DeclareMathOperator{\Ext}{Ext_\Z^1}
\DeclareMathOperator{\coker}{coker}
\DeclareMathOperator{\stab}{Stab}

\newcommand{\dual}{\star}             
\newcommand{\DG}{\mathrm{DG}}         

\begin{document}

\title[Inertia groups of a toric DM stack, fake weighted projective stacks, and labelled sheared simplices]{Inertia groups of a toric Deligne-Mumford stack, fake weighted projective stacks, and labelled sheared simplices}

\author{Rebecca Goldin}
\address{Mathematical Sciences MS 3F2 \\
 George Mason University\\
 4400 University Drive \\
Fairfax, VA 22030 \\ USA}
\email{rgoldin@math.gmu.edu}
\urladdr{\url{http://math.gmu.edu/~rgoldin/}}

\author{Megumi Harada}
\address{Department of Mathematics and
Statistics\\ McMaster University\\ 1280 Main Street West\\ Hamilton, Ontario
L8S4K1\\ Canada}
\email{Megumi.Harada@math.mcmaster.ca}
\urladdr{\url{http://www.math.mcmaster.ca/Megumi.Harada/}}
\thanks{MH is partially supported by an NSERC Discovery Grant,
an Ontario Ministry of Research
and Innovation Early Researcher Award, and 
a Canada Research Chair (Tier 2) award. RG is partially supported by NSF Disciplinary Grant \#202726.}

\author{David Johannsen}
\address{Mathematical Sciences \\
 George Mason University\\
 4400 University Drive \\
Fairfax, VA 22030 \\ USA}
\email{djohann1@gmu.edu}

\author{Derek Krespki}
\address{Department of Mathematics \\
University of Manitoba\\ 342 Machray Hall \\ Winnipeg, Manitoba R3T 2N2\\ Canada}
\email{Derek.Krepski@umanitoba.ca}
\urladdr{\url{http://server.math.umanitoba.ca/~dkrepski/}}


\date{}

\maketitle

\begin{abstract}  
  This paper determines the inertia groups (isotropy groups) of the
  points of a toric Deligne-Mumford stack $[Z/G]$ (considered over the category of smooth
  manifolds) that is realized from a quotient construction using 
  a stacky fan or stacky polytope. The computation provides an explicit
  correspondence between certain geometric and combinatorial data.  In
  particular, we obtain a computation of 
  the connected component of the identity
  element $G_0 \subset G$ and the component group $G/G_0$ in terms
  of the underlying stacky fan, enabling us to characterize the toric DM stacks which are global quotients. As another application, we obtain a
  characterization of those stacky polytopes that yield stacks
  equivalent to weighted projective stacks and, more generally, to 
  \emph{`fake' weighted projective stacks}.  Finally, 
we illustrate our results in
  detail in the 
  special case of \emph{labelled sheared simplices}, where explicit
  computations can be made in terms of the facet labels. 
\end{abstract}

\section*{Introduction}

Toric varieties have been studied for over 35 years. They provide an elementary but illustrative class of examples in algebraic geometry, while also offering insight into related fields such as integrable systems and combinatorics, where the corresponding combinatorial object is a \emph{fan}.
In their foundational paper \cite{BCS05},
Borisov, Chen, and Smith introduce the notion of a \emph{stacky fan},
the combinatorial data from which one constructs toric
  Deligne-Mumford (DM) stacks, which are the stack-theoretic analogues of classical toric varieties. 
When the corresponding fan is polytopal, classical toric varieties have been studied from the perspectives of both algebraic and symplectic geometry. Similarly, when the underlying fan of a stacky fan is polytopal, a toric DM stack admits a description in the language of 
symplectic geometry via the combinatorial data of a \emph{stacky polytope} introduced by Sakai \cite{Sakai2010}. (In the symplectic-geometric context---and particularly in this manuscript---stacks are considered over  the category $\mathsf{Diff}$ of smooth manifolds.) This subfamily of toric DM stacks can be viewed as a generalization of Lerman and Tolman's toric orbifolds associated to \emph{labelled polytopes} \cite{LT97} (cf. Section \ref{definition:labelled polytope} for details), which are in turn a generalization of the Delzant polytopes that classify symplectic toric manifolds. Thus toric DM stacks are a generalization 
 of smooth toric varieties to not-necessarily-effective orbifold toric varieties (including weighted projective spaces), and
 provide a fertile ground for exploration of stacks via this large class of examples.

The exposition in this article is intended to be accessible to a wide audience, including researchers who are not experts in this area.  The tools we develop allow one to compute concretely 
the isotropy groups of toric DM stacks without reference to much stack-theoretic machinery. Moreover, 
we include many detailed examples illustrating our results.

The mathematical contributions of this manuscript are as follows. 
We first describe in Theorem ~\ref{theorem:isotropy}, Proposition~\ref{prop:concrete isotropy}, and Proposition~\ref{prop:compactisomorphism} an explicit computation of the isotropy groups of toric DM stacks, realized as quotient stacks $[Z/G]$ for appropriate space $Z$ and abelian Lie group $G$, in terms of the combinatorial data (i.e. stacky fan) determining the toric DM stack. 

Secondly, as an application of our 
description of isotropy groups of toric DM stacks, we give a computation of the connected component of the identity element 
$G_0 \subset G$ and the component group $G/G_0$ in terms
  of the underlying stacky fan (Proposition~\ref{prop:stackyfanG0}, Lemma~\ref{le:GmodG0}, Proposition~\ref{eq:explicitGmodG0}). To place this computation into context,  recall that 
a stack is called a \emph{global quotient} if it is equivalent to 
a quotient stack $[M/\Lambda]$ where $\Lambda$ is a finite group acting on a manifold $M$. Stack invariants of global quotients
are simpler to compute than for general stacks. Thus, given a stack $\mathcal{X}$, it is an interesting problem to determine whether or not it is 
a global quotient.
In the case of toric DM stacks, this problem is discussed in \cite{HaradaKrepski:2011}, where it is shown that 
a toric DM stack is a global quotient if and only if the restriction of the $G$-action on $Z$ to the connected component of the identity $G_0 \subset G$ is a free action. Moreover, in this case, one may choose the finite group to be $\Lambda=G/G_0$, acting on the quotient  $M=Z/G_0$, which is indeed a manifold, provided $G_0$ acts freely on $Z$ (so $[Z/G_0]$ is the universal cover, in the sense of stacks, of $[Z/G]$). 
In this manuscript, our computation of isotropy groups leads to a 
characterization of those toric DM stacks that are (stacks equivalent to) global quotients of a finite group action and to a description of its universal cover (cf. Section~\ref{sec:globquot}). 

Our third set of results concern weighted projective stacks (resp. fake weighted projective stacks), which are natural stack-theoretic analogues of the classical weighted projective spaces (resp. fake weighted projective spaces as considered in \cite{Bu02, Ka09}). These form a rich class of 
examples that have been 
studied extensively both as stacks and as orbifolds (e.g. see \cite{BoissiereMannPerroni:2009b, Jiang:2007, Mann:2008} among others). As another application of our computation of isotropy groups,
in Proposition~\ref{prop:WPS} (resp. Proposition \ref{prop:fakeWPS})  
we obtain an exact
  characterization of those stacky polytopes which yield (stacks
  equivalent to) weighted projective stacks (resp. 
  fake weighted projective stacks).  

Finally, in Section~\ref{se:shimplexWPS} we introduce a class of labelled polytopes, which we call \emph{labelled sheared simplices}.  
These are 
labelled simplices with all facets but one lying on coordinate hyperplanes.  In this special case
we  illustrate the aforementioned results concretely in terms of the facet labels.

\bigskip

\noindent {\bf Acknowledgements.} We would like to thank Yael Karshon, Graham Denham and Jim Lawrence for useful discussions on both toric varieties and convex polytopes.

\section{Preliminaries}\label{se:prelim}

In this section we recall some background regarding stacky fans and polytopes
and their associated toric Deligne-Mumford stacks.
We assume some familiarity with stacks, in particular their use in modelling  group actions on manifolds and orbifolds.  
We refer the
reader to e.g. \cite{Fantechi:2001, Edidin:2003} and references
therein for basic definitions and ideas in the theory of stacks. Within the
field of algebraic geometry there
is by now an extensive literature on (algebraic) stacks (see e.g. the
informal guide \cite{Alper-guide}), but in other categories
(e.g. $\mathsf{Diff}$ or $\mathsf{Top}$) the literature continues to develop. 
The present authors learned a great deal from the unpublished
(in-progress) notes 
\cite{BCEFFGK-stacks} as well as \cite{Lerman:2010} and \cite{Metzler03}.

The stacks appearing in this paper are quotient stacks
over $\mathsf{Diff}$  associated to  smooth, proper, locally-free Lie group actions on manifolds.  Such quotient stacks are Deligne-Mumford (i.e. stacks that admit a presentation by a proper \'etale Lie groupoid, cf. \cite[Theorem 2.4]{LermanMalkin2009}) and therefore model smooth orbifolds.  
  Often in the algebraic literature, the term orbifold is used for DM stacks with trivial generic stabilizer (e.g. as in \cite{FantechiMannNironi:2010}), what we shall refer to  as an \emph{effective} orbifold.  

The class of  DM stacks we work with are \emph{toric DM stacks}, arising from the combinatorial data of a \emph{stacky fan} \cite{BCS05}.  Our original motivation was to work instead with \emph{stacky polytopes}, the symplectic counterparts of stacky fans, which give rise  to  \emph{symplectic toric DM stacks} \cite{Sakai2010}.   These offer a modern perspective on the symplectic toric orbifolds of Lerman and Tolman \cite{LT97} constructed from \emph{labelled polytopes} (see Section \ref{definition:labelled polytope} below).  However, since our results do not depend on (or make use of) the symplectic structure that results from this perspective, we choose to work mainly with stacky fans.

\subsection{Stacky fans and polytopes}

Mainly to establish notation, we  briefly recall some basic definitions of the combinatorial data appearing in the above discussion. We use $( - )^\dual$ to denote the functor $\mathrm{Hom}_{\Z}(-,\Z)$ or $\mathrm{Hom}_\R(-,\R)$; it should be clear from context which one is meant.  Let $\T$ denote the group of units $\C^\times$, and $\mu_k \subset \T$ the cyclic group of $k$-th roots of unity.
Let $\{ e_1, \ldots, e_n\}$ be the standard basis vectors in $\Z^n \subset \R^n$.

\begin{definition} \cite{BCS05} \label{definition:stacky fan}
A  \emph{stacky fan} is a triple $(N, \Sigma, \beta)$ consisting of a rank~$d$
finitely generated Abelian group~$N$, a rational simplicial fan $\Sigma$ in
$N\otimes \R$ with rays $\rho_1, \ldots, \rho_n$ and a homomorphism $\beta:\Z^n
\to N$ satisfying:
\begin{enumerate}
\item \label{item:rays span} the rays $\rho_1, \ldots, \rho_n$ span $N\otimes \R$, and
\item \label{item:beta surjects onto rays} for $1\leq j \leq n$, $\beta(e_j)\otimes 1$ is on the ray $\rho_j$.
\end{enumerate}
\end{definition}

Given a polytope $\Delta \subseteq \R^d$, recall that the
corresponding fan $\Sigma = \Sigma(\Delta)$ is obtained by setting the one dimensional cones $\Sigma^{(1)}$ to be the positive rays spanned by the inward-pointing normals to the facets of
$\Delta$; a subset $\sigma$ of these rays is a cone in $\Sigma$
precisely when the corresponding facets intersect nontrivially in
$\Delta$. Observe that under this correspondence, facets intersecting in a vertex of $\Delta$ yield maximal cones (with respect to inclusion) in $\Sigma(\Delta)$.

\begin{definition}\label{definition:stacky polytope} \cite{Sakai2010}
 A \emph{stacky polytope} is a triple $(N, \Delta, \beta)$ consisting of a
rank~$d$ finitely generated Abelian group~$N$, a simple polytope
$\Delta$ in $(N\otimes  \R)^\dual$ with $n$ facets $F_1, \ldots, F_n$ and a
homomorphism $\beta:\Z^n \to N$ satisfying:
\begin{enumerate}
 \item \label{item:finite cokernel} the cokernel of $\beta$ is finite, and
 \item \label{item:beta hits normals} for $1\leq j \leq n$,  $\beta(e_j) \otimes  1$ in
$N\otimes  \R$ is an inward pointing normal to the facet $F_j$.
\end{enumerate}
\end{definition}

Condition \ref{item:beta hits normals} above implies that the polytope $\Delta$ in Definition \ref{definition:stacky polytope} is a rational polytope.  Also, from   the preceding
discussion it follows immediately that the data of a stacky polytope
$(N,\Delta, \beta)$ specifies the data of a stacky fan by the
correspondence $(N,\Delta, \beta) \mapsto (N, \Sigma(\Delta),\beta)$.
Indeed,  $\Delta$ is simple if and only if $\Sigma(\Delta)$ is simplicial.  Moreover, the fan $\Sigma(\Delta)$ is rational by condition  \ref{definition:stacky polytope} (\ref{item:beta hits normals}).  Finally, $(N,\Delta,\beta)$ satisfies 
conditions (\ref{item:finite cokernel}) and
(\ref{item:beta hits normals}) of Definition \ref{definition:stacky polytope} if and only if 
$(N,\Sigma(\Delta),\beta)$ satisfies  conditions (\ref{item:rays span}) and (\ref{item:beta surjects onto rays})
of Definition \ref{definition:stacky fan}.

The extra information encoded in a stacky polytope $(N,\Delta,\beta)$ (compared with the stacky fan $(N,\Sigma(\Delta),\beta)$) results in a symplectic structure on the associated toric DM stack.  
Given a presentation of a rational polytope $\Delta$ as the intersection of half-spaces
\begin{align} \label{equation:polytope is rational}
 \Delta = \bigcap_{i=1}^n \left\{ x \in (N \otimes  \R)^\dual \,|\,  
x( \beta(e_{i})\otimes 1 )  \geq -c_i \right\}
\end{align}
for some $c_i \in \R$ and where each $\beta(e_{i})\otimes 1\in
N \otimes \R$  is the inward pointing normal to the facet $F_i$, the
fan $\Sigma(\Delta)$ only retains the data of the positive ray spanned
by the normals, and not the parameters $c_i$, which encode the symplectic structure on the resulting DM stack (see \cite{Sakai2010} for details).

\bigskip

Recall (as in \cite{BCS05}) that given a stacky fan
$(N,\Sigma,\beta)$, the corresponding DM stack may be constructed as a quotient stack $[Z_\Sigma/G]$ as follows.   As with classical toric varieties, the fan $\Sigma$ determines an ideal 
$$J(\Sigma)=\langle \prod_{\rho_i \not\subset \sigma} z_i \, : \, \sigma \in \Sigma \rangle \subset \C[z_1, \ldots, z_n].$$ 
 Let $Z_\Sigma$ denote the complement $\C^n \smallsetminus V(J(\Sigma))$ of the vanishing locus of $J(\Sigma)$.  
Next, we recall a certain group action on $Z_\Sigma$.

Choose a free resolution
$$
0\to \Z^\ell \stackrel{Q}{\longrightarrow} \Z^{d+\ell} \to N \to 0
$$  
of the $\Z$-module $N$, and let $B:\Z^n \to \Z^{d+\ell}$ be a lift of
$\beta$.  With these choices, define the \emph{dual group} $\DG(\beta) = (\Z^{n+\ell})^\dual/ \im [B\,
Q] ^\dual$ where $[B\,Q]:\Z^{n+\ell}=\Z^{n}\oplus \Z^\ell \to \Z^{d+\ell}$
denotes the map whose restrictions to the first and second summands
are $B$ and $Q$, respectively.  Let $\beta^\vee: (\Z^n)^\dual \to
\DG(\beta)$ be the composition of the inclusion $(\Z^n)^\dual \to
(\Z^{n+\ell})^\dual$ (into the first $n$ coordinates) and the quotient
map $(\Z^{n+\ell})^\dual \to \DG(\beta)$. Applying the functor
$\Hom_\Z(-,\T)$ to $\beta^\vee$  yields a homomorphism
$G:=\Hom_\Z(\DG(\beta),\T) \to \T^n$, which defines a $G-$action
on $\C^n$ that leaves $Z_\Sigma \subset \C^n$ invariant.  Define $\mathcal{X}(N,\Sigma,\beta)=[Z_\Sigma/G]$.  By Proposition 3.2 in \cite{BCS05}, $\mathcal{X}(N,\Sigma,\beta)$ is a  DM stack.  At times, we shall simply use the notation $[Z_\Sigma/G]$ to denote $\mathcal{X}(N,\Sigma,\beta)$.

The above construction was adapted to stacky polytopes by Sakai in \cite{Sakai2010}.   As the reader may verify, the DM stack $\mathcal{X}(N,\Delta,\beta)$ obtained from a stacky polytope is a quotient stack obtained by symplectic reduction
$[\mu^{-1}(\tau)/K]$ where $\mu^{-1}(\tau)\subset Z_{\Sigma(\Delta)} \subset \C^n$ is a certain level set of a moment map $\mu:\C^n \to \mathfrak{k}^\dual$ for a Hamiltonian action of $K:=\Hom_\Z(\DG(\beta),S^1)$ on $\C^n$.  
By \cite[Theorem 24]{Sakai2010}, the quotient stacks $[Z_\Sigma/G]$ and $[\mu^{-1}(\tau)/K]$ are equivalent.

\begin{example}
 \label{ex:GconnNfreeWPS} 
Consider the stacky polytope  $(N, \Delta, \beta)$, with $N= \Z^2$, $\Delta$ the simplex in $\R^2 \cong (N \otimes \R)^*$ given by the convex hull of
$(0,0)$, $(0,1)$ and $(1,0)$,
and $\beta: \Z^3\to N$ given by the matrix
$$
\beta=\begin{bmatrix}
-2 & 3 & 0 \\
-2 & 0 & 5
\end{bmatrix}.
$$

 The corresponding stacky fan $(N,\Sigma,\beta)$ is then given by the same $N$ and $\beta$, and $\Sigma =\Sigma(\Delta)$ the fan dual to $\Delta$ (see Figure \ref{figure:WPS}). 
A convenient way to represent the homomorphism $\beta$ is to use ray or facet labels (see Section \ref{definition:labelled polytope}), as in Figure \ref{figure:WPS}.

\setlength{\unitlength}{2cm}
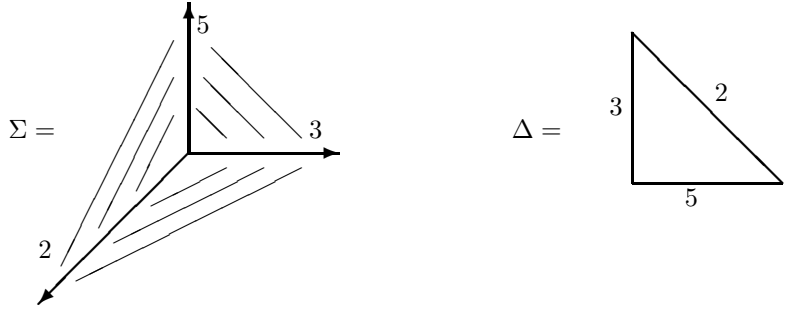
\begin{figure}[h]
\centering
\begin{minipage}{.4\textwidth}
\centering
\begin{picture}(2.125,2.125)(0,-.7)
\put(-.2,.3){$\Sigma=$}
\thicklines
\put(1,.2){\vector(1,0){1}}\put(1.05,1){$5$}
\put(1,.2){\vector(0,1){1}}\put(1.8,.3){$3$}
\put(1,.2){\vector(-1,-1){1}}\put(0,-.5){$2$}
\thinlines
\put(1.25,.3){\line(-1,1){.2}}
\put(1.5,.3){\line(-1,1){.4}}
\put(1.75,.3){\line(-1,1){.6}}
\put(.9, .45){\line(-1,-2){.25}}
\put(.9, .7){\line(-1,-2){.5}}
\put(.9,.95){\line(-1,-2){.75}}
\put(1.25, .1){\line(-2,-1){.5}}
\put(1.5, .1){\line(-2,-1){1}}
\put(1.75, .1){\line(-2,-1){1.5}}
\end{picture} 
\end{minipage}
\begin{minipage}{.4\textwidth}
\centering
\begin{picture}(2.125,2.125)(0,-.7)
\put(-0.2,.3){$\Delta=$}
\thicklines
\put(.6,0){\line(1,0){1}}\put(.95,-.15){$5$}
\put(.6,0){\line(0,1){1}}\put(.45,.45){$3$}
\put(.6,1){\line(1,-1){1}}\put(1.15,.55){$2$}
\end{picture} 
\end{minipage}
\caption{A polytope $\Delta$ and its dual fan $\Sigma=\Sigma(\Delta)$.  The labels on the facets of $\Delta$ (resp. ray generators of $\Sigma$) encode the homomorphism $\beta:\Z^3 \to N$.}
 \label{figure:WPS}
\end{figure}

To compute the corresponding DM stack $[Z_\Sigma/G]$, note that $Z_\Sigma=\C^3 \smallsetminus \{0\}$. We find
$
 \DG(\beta) = (\Z^3)^\dual / \im \beta^\dual \cong \Z,
$
where the isomorphism may be chosen as $\bar{f}([a,b,c]) = 15a+10b+6c$.  Therefore, $G=\Hom(\DG(\beta),\T) \cong \T$.  Since the map $\beta^\vee$ is simply the projection $f:(\Z^3)^\dual \rightarrow \DG(\beta)\cong \Z$, where $f(a,b,c) = 15a+10b+6c$, the homomorphism $G\rightarrow \T^3$ induced by $\beta^\vee$ is then $t\mapsto (t^{15},t^{10},t^{6})$.
It follows that the corresponding stack is equivalent to a weighted projective stack, $\calP(15,10,6)$.
\end{example}

We modify the above example to illustrate the construction for a $\Z$-module $N$ with torsion.

\begin{example} \label{eg:WPSglobalisotropy}
Let $N=\Z^2\oplus \Z/2\Z$, and let $\Sigma$ be the fan in Figure \ref{figure:WPS}. Set $\beta:\Z^3 \to N$ to be $$\beta(x,y,z)=(-2x+3y,-2x+5z,x+y+z \mod 2).$$
As in Example \ref{ex:GconnNfreeWPS}, $Z_\Sigma = \C^3 \smallsetminus\{0\}$.  To compute $G$, we choose the resolution $0 \to \Z \stackrel{Q}{\longrightarrow} \Z^3 \to N \to 0$ with $Q=\begin{bmatrix}
0& 0 & 2
\end{bmatrix}^T$, and choose $B$ so that  
$$
[B\, Q] = \begin{bmatrix}
-2 & 3 & 0 & 0 \\
-2 & 0 & 5 & 0 \\
-1 & 1 & 1 & 2
\end{bmatrix}.
$$
Therefore, $\DG(\beta) = (\Z^4)^\dual / \im [B\, Q]^\dual \cong \Z$, where the isomorphism can be chosen as $\bar{f}([a,b,c,d])=30a+20b+12c-d$.  Therefore, $G\cong \T$, and $\beta^\vee:(\Z^3)^\dual \to \DG(\beta)\cong \Z$ is given by $f(a,b,c)=30a+20b+12c$.  Hence the homomorphism $G\to \T^3$ describing the $G$-action on $Z_\Sigma$ is $t\mapsto (t^{30},t^{20},t^{12})$.  It follows that the corresponding stack is equivalent to a weighted projective stack $\calP(30,20,12)$, which has global stabilizer $\mu_2$.
\end{example}

\subsection{Relation with Delzant's construction and labelled polytopes}
\label{definition:labelled polytope}

The construction of the quotient stack $[\mu^{-1}(\tau)/K]$ from a stacky polytope may be viewed as a generalization of Lerman-Tolman's generalization \cite{LT97} of the Delzant construction, which we review next.

In its original form \cite{LT97},  a \emph{labelled polytope} is a pair $(\Delta,
\{m_i\}_{i=1}^n)$ consisting of a convex simple polytope $\Delta$ in $V^\dual$, where $V$ is a real vector space,  
 with $n$ facets $F_1, \ldots, F_n$ whose relative interiors are
labelled with positive integers $m_1, \ldots, m_n$. The polytope is assumed to be rational with respect to a chosen lattice $N\subset V$.
Identifying $V\cong N\otimes \R$, we may denote  the primitive inward pointing normals  by $\nu_1\otimes 1, \ldots, \nu_n\otimes 1$. Then defining
$\beta:\Z^n \to N$ by the formula   $\beta(e_i) = m_i \nu_i$ realizes $(N,\Delta,\beta)$ as a stacky polytope. Furthermore, any stacky polytope with $N$ free can be realized as a labelled polytope. Thus  labelled polytopes are precisely those
stacky polytopes for which the $\Z$-module $N$ is  a free
module.

Given a labelled polytope $(\Delta, \{ m_i \}_{i=1}^{n})$ in $(N\otimes \R)^\dual \cong(\R^d)^\dual$, we may proceed with the Delzant construction to obtain a quotient stack $[\mu^{-1}(\tau)/\KD]$ as a symplectic reduction, where the group $\KD \subset (S^1)^n$ acts via the standard linear $(S^1)^n$-action on $\C^n$.  As we shall see below  $K_D$ is isomorphic to the group $K=\Hom(\DG(\beta),S^1)$ arising from Sakai's construction.  Furthermore, it is straightforward to verify that the isomorphism is compatible with the respective group actions of $K_D$ and $K$ on $\C^n$, and thus the symplectic quotient stacks $[\mu^{-1}(\tau)/K_D]$ and $[\mu^{-1}(\tau)/K]$ are equivalent.

The group $\KD$ in the Delzant construction is defined as follows. Let $\beta:\Z^n \to N$ be given by $\beta(e_i)=m_i\nu_i$, where $m_i\nu_i$ are the weighted normals to the facets of $\Delta$, and consider the resulting homomorphism $\bar\beta:(S^1)^n \to (S^1)^d$ induced by $\beta_\R=\beta\otimes \R$ (where we have chosen  identifications $N\cong \Z^d$, $S^1\cong \R/\Z$). Define $\KD =\ker \bar\beta$.

To compare the groups $\KD$ and $K$, we note that since $N$ is free it is easy to verify that $\DG(\beta)=\coker \beta^\dual$, and hence we have the short exact sequence
$$
0 \to N^\dual \to (\Z^n)^\dual \to \DG(\beta) \to 0
$$
which yields the short exact sequence,
$$
1 \to K \to \Hom((\Z^n)^\dual,S^1) \to \Hom(N^\dual,S^1) \to 1.
$$
Using  the natural isomorphism $\Hom(M^\dual,S^1) \cong M\otimes S^1$ (for a free $\Z$-module $M$) and identifications 
$$
\Z^n\otimes S^1\cong (S^1)^n \quad \text{and} \quad N\otimes S^1\cong (S^1)^d
$$ resulting from the chosen identification $N\cong \Z^d$, we may readily identify $K\cong \KD$. 

\begin{example}
\label{eg:segment}
Let $N=\Z$. Consider the labelled polytope $\Delta$ in $\R\cong (N\otimes \R)^\dual $ consisting of a line segment with labels $r$ and $s$ at each endpoint (see Figure \ref{fig:segment}). 
\setlength{\unitlength}{1mm}
\begin{figure}[h]
\begin{picture}(15,7)
\thicklines
\put(10,2){\line(1,0){10}}\put(7,2){$r$} \put(22,2){$s$}
\put(10,2){\circle*{1}}\put(20,2){\circle*{1}}
\end{picture}
\caption{A labelled polytope $\Delta$ in $\R$.}
\label{fig:segment}
\end{figure}
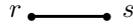

The homomorphism $\beta:\Z^2 \to N$ in the corresponding stacky polytope (and stacky fan) is given by the matrix $\beta= \begin{bmatrix}
-s & r
\end{bmatrix}$. Let $g=\gcd(r,s)$.  Then $\DG(\beta)=(\Z^2)^\dual / \im \beta^\dual \cong \Z \oplus \Z /g\Z$, which may be realized by the isomorphism $\bar{f}([a,b])=(\frac{r}{g}a+\frac{s}{g}b, -ya+xb \mod g)$ where $x$ and $y$ are integers satisfying $g=xr+ys$.  It follows that $G\cong \T \times \mu_g$, where $\mu_g \subset \T$ is the cyclic group of $g$-th roots of unity.  Note that $G$ is connected if and only if $\gcd(r,s)=1$.  

Under the above identification, the homomorphism $\beta^\vee:(\Z^2)^\dual \to \DG(\beta) \cong \Z \oplus \Z/g\Z$ is simply the projection $f(a,b)=(\frac{r}{g}a+\frac{s}{g}b, -ya+xb \mod g\Z)$.  Therefore, the homomorphism $\T\times \mu_g \cong G\to \T^2$ that determines the action on $Z_\Sigma =\C^2 \smallsetminus \{0\}$ is then given by $(t,\xi_g^k) \mapsto (t^{\frac{r}{g}} \xi_g^{-ky},t^{\frac{s}{g}} \xi_g^{kx})$ where $\xi_g \in \C$ denotes a primitive $g$-th root of unity.
\end{example}

\section{Isotropy and Stacky Fans}\label{se:isotropy}

Our goal in this section is to compute the local isotropy group of each point of a toric DM stack $\mathcal{X}(N,\Sigma,\beta)=[Z_{\Sigma}/G]$ by computing the subgroup $\stab(z) \subset G$ that fixes a given point $z\in Z_{\Sigma}$.  The main result, Theorem \ref{theorem:isotropy} in Section \ref{sec:generalisotropy}, describes all possible isotropy groups that arise.  A discussion of the connected component of $G$ and its role in detecting global quotient stacks appears in \ref{sec:globquot}, along with further details in \ref{sec:Nfree} for the case of labelled polytopes.

\subsection{Isotropy and stacky fans} \label{sec:generalisotropy}

Recall that $Z_\Sigma$ is defined as the complement in $\C^n$ of the zero-set of the ideal 
$
J(\Sigma) 
$, which is described in more detail next.

For $\sigma\in \Sigma,$ let $I_\sigma = \{i: \rho_i\subset \sigma\}$, and $J_\sigma$ its complement.  Then 
\begin{align*}
V(J(\Sigma)) &= \bigcap_{\sigma\in\Sigma} \{(z_1,\dots, z_n) \mid \prod_{\rho_i\not\subset \sigma} z_i = 0\} 
\mbox{ implying}\\
Z_{\Sigma}=\C^n\smallsetminus V(J(\Sigma)) &=\bigcup_{\sigma\in\Sigma} \{(z_1,\dots, z_n) \mid z_i \neq 0 \mbox{ whenever }i\in J_\sigma\}\\
& = \bigcup_{\sigma \in \Sigma} \{(z_1,\dots, z_n) \mid I_{z}\subset I_\sigma \}
\end{align*}
where $I_z= \{i \mid z_i=0\}$.

There is a natural decomposition of  $Z_\Sigma$ since an inclusion of cones $\sigma'\subset \sigma$ yields an inclusion $Z_{\sigma}\subset Z_{\sigma'}$, where $Z_\sigma :=\{(z_1,\dots, z_n): I_{z} \subset I_\sigma\}$. Furthermore, for any $z=(z_1,\dots, z_n)\in Z_\sigma$, there is a cone $\sigma_z\subset \sigma$ given by the span of  the minimal generators of the rays $\rho_i$ with $i\in I_z$. This follows from our assumption that $\Sigma$ is simplicial. Since the number of rays $\rho_i$ with $i\in I_\sigma$ equals the dimension of $\sigma$, any subset of these rays spans a face of $\sigma$ and is thus in the fan $\Sigma$. It follows, then, that for every point $z\in Z_\Sigma$, we may write $z\in Z_{\sigma_z}$ where the cone $\sigma_z$ satisfies $\{i: z_i=0\} = I_{\sigma_z}$.  Moreover, $\sigma_z$ is \emph{minimal} in the sense that $\sigma_z \subset \sigma$ for any $\sigma$ such that $z\in Z_\sigma$.

For a point $z \in \C^n$, the subgroup  in $\T^n$  fixing $z$ is $\{(t_1,\dots, t_n): t_i=1 \mbox{ if } z_i\neq 0\}$, which motivates the following definition. For any subset $I\subset \{1,\dots, n\}$ and its complement $J$, let
$$
\T^I = \{(t_1, \ldots, t_n) : i \in J \mbox{ implies } t_i =1\} \subseteq \T^n.
$$
Note that $\T^{I}$ is the kernel of the map $\T^n \rightarrow \T^{|J|}$ given by projection onto the coordinates indicated by $J$ with cardinality $|J|$. 

Since $G$ acts on $Z_\Sigma$ via the homomorphism $G\to \T^n$ induced by $\beta^\vee$,  then the isotropy $\stab(z)$ associated to a point $z\in Z_\Sigma$ is given by the kernel of the map 
\begin{equation}
\label{eq:kernel}
G \longrightarrow \T^n \overset{\pi}\longrightarrow \T^{|J_z|} 
\end{equation}
where $J_z$ is the complement of $I_z$.  At times it is useful to view the subset $I_z$ as $I_\sigma$ for the cone $\sigma=\sigma_z$ in $\Sigma$ described above, in which case we denote $\stab(z)$ by $\Gamma_\sigma$.

\begin{remark} \label{remark:trivial isotropy and max cones}
  Notice that an inclusion of cones $\sigma'  \subset \sigma$ in $\Sigma$ induces an inclusion $J_{\sigma} \subset J_{\sigma'}$ and hence the following commutative diagram,
$$
\xymatrix{
{\T^n} \ar[dr] \ar[r] & {\T^{|J_{\sigma'}|}} \ar[d] \\
& {\T^{|J_\sigma|}}
}
$$
where the vertical map is the natural projection.
Therefore, there is a natural inclusion of isotropy groups $\Gamma_{\sigma'} \subset \Gamma_\sigma$.  It follows each such isotropy group is contained in $\Gamma_\sigma$ for some maximal cone $\sigma$.  In particular, all isotropy groups are trivial if and only if $\Gamma_\sigma$ is trivial for maximal cones $\sigma$ in $\Sigma$.
\end{remark}

\begin{theorem} \label{theorem:isotropy}
Let $(N,\Sigma,\beta)$ be a stacky fan and $[Z_\Sigma/G]$ its corresponding toric DM stack.  For a point $z~=~(z_1, \ldots, z_n)$ in $Z_\Sigma \subset \C^n$, let $N_z \subset N$ denote the submodule generated by $\{\beta(e_i) \mid z_i=0\}$. The isotropy group $\stab(z)$ is  isomorphic to $\Hom(\Ext(N/N_z,\Z),\T)$; therefore, $\stab(z)$  is (non-canonically) isomorphic to the torsion submodule $\Tor(N/N_z)$.
\end{theorem}
\begin{proof}
Let $\sigma$ be a cone in $\Sigma$ with $I_z=I_\sigma$.
We noted above that the stabilizer of $z$ is given by  $\Gamma_\sigma$, the kernel of the composition \eqref{eq:kernel}.
 This composition is realized by applying the functor $\Hom(-,\T)$ to the composition $f=\beta^\vee\circ\pi^\dual$
$$
(\Z^{|J_\sigma|})^\dual\overset{\pi^\dual}\longrightarrow  (\Z^n)^\dual \overset{\beta^\vee}\longrightarrow \DG(\beta),
$$
where $\pi^\dual$ is inclusion of the relevant factors. Moreover, since $\T$ is injective as a $\Z$-module,  the kernel of \eqref{eq:kernel} is $\Hom(\coker f,\T)$.  As we shall see next, 
 $\coker f \cong \Ext(N/N_\sigma,\Z)$, where $N_\sigma = N_z \subset N$ denotes the subgroup generated by $\{\beta(e_i) \mid i \in I_\sigma\}$.

Let $\beta_\sigma:\Z^{I_\sigma} \to N_\sigma$ denote the restriction of $\beta$ to $\Z^{I_\sigma}$ together with its codomain.
 We claim that $\beta_\sigma: \Z^{I_\sigma} \to N_\sigma$ is an isomorphism, and hence $N_\sigma$ is free.
That $\Sigma$ is simplicial means that the $\{\beta(e_i) \otimes 1\}_{i\in I_\sigma}$ are linearly independent in $N\otimes \R$.  Therefore, $\rank N_\sigma = |I_\sigma|$.  Since $\beta_\sigma$ is a surjective homomorphism of modules of the same rank, $\beta_\sigma \otimes \R$ is an isomorphism of vector spaces.  But since the domain $\Z^{I_\sigma}$ of $\beta_\sigma$  is free,  $\beta_\sigma$ must be injective as well. This verifies the claim.

In particular, this implies that $\DG(\beta_\sigma)$ is trivial.  Any lift $B_\sigma$ of $\beta_\sigma$ is an isomorphism and 
$N_\sigma$ has no torsion, so $\DG(\beta_\sigma)=\coker[B_\sigma]^\dual=\coker \beta_\sigma^\dual$.

Consider the following diagram, whose rows are exact.
$$
\xymatrix{
0 \ar[r] & \Z^{|I_\sigma|} \ar[r] \ar[d]_{\beta_\sigma} & \Z^n \ar[r] \ar[d] _{\beta} &  \Z^{|J_\sigma|} \ar[r] \ar[d]_{\beta_J} &  0\\
0 \ar[r] & N_\sigma \ar[r] & N \ar[r] & N/N_\sigma \ar[r] & 0
}
$$
By Lemma 2.3 in \cite{BCS05}, we get the following commutative diagram with exact rows, noting that $\DG(\beta_\sigma)$ is trivial.
$$
\xymatrix{
 0 \ar[r] &  (\Z^{|J_\sigma|})^\dual \ar^{\pi^\dual}[r] \ar[d] &  (\Z^n)^\dual \ar[r] \ar[d]_{\beta^\vee} & (\Z^{|I_\sigma|})^\dual \ar[r] \ar[d] &  0 \\
 0\ar[r] & \DG(\beta_J) \ar[r]^{\cong} &  \DG(\beta) \ar[r] & 0 &
}
$$
We identify $f=\beta^\vee\circ\pi^\dual$ with the left vertical arrow.

Applying the exact sequence (2.0.3) from \cite{BCS05} to $\beta_J:\Z^{|J_\sigma|} \to N/N_\sigma$, we get
$$
(N/N_\sigma)^\dual \to (\Z^{|J_\sigma|})^\dual \to \DG(\beta_J) \to \Ext(N/N_\sigma,\Z) \to 0
$$
whence $\coker(f)\cong \Ext(N/N_\sigma,\Z)$, as required.
Since $\Ext(N/N_\sigma,\Z)$ is (non-canonically) isomorphic to (the finite abelian group) $\Tor(N/N_\sigma)$, it follows that $\Hom(\Ext(N/N_\sigma,\Z),\T)$ is isomorphic to $\Tor(N/N_\sigma)$. 

\end{proof}

\begin{example} \label{eg:WPSisotropycalc}
Consider the toric DM stack from Example~\ref{ex:GconnNfreeWPS}. We  compute the isotropy for the points of the form $z=(0,a,0)$ and $w=(0,a,b)$ in $Z_\Sigma$ with $a,b \neq 0$.  Since $I_{z}=\{1,3\}$, then $N_z \subset N$ is the subgroup generated by $(0,5)$ and $(-2,-2)$.  Therefore, $N/N_z \cong \Z/5\Z \oplus \Z/2\Z$, and $\stab(z)\cong\Tor(N/N_z)\cong \mu_{10}$.

Since $I_{w}=\{1\}$, then $N_w$ is the subgroup generated by $(-2,-2)$, and $N/N_w \cong \Z\oplus \Z/ 2\Z$.  Therefore, $\stab(w) \cong \Tor(N/N_w) \cong \mu_{2}$.  Note that the isotropy for $w$ can simply be read off from the corresponding facet label (see Figure \ref{figure:WPS}) in this case. For higher dimensional cones, a more detailed analysis is required---e.g. see Section~\ref{se:2dimsheared}.
\end{example}

The proof of Theorem~\ref{theorem:isotropy} does not show explicitly how $\Tor(N/N_z)$ may be viewed as a subgroup of $G$.   
Our next goal is Proposition~\ref{prop:concrete isotropy} which gives an explicit 
identification of the stabilizer group and $\Tor(N/N_z)$. To accomplish this, we first construct a map $\Tor(N/N_z) \to G$ \eqref{concrete inclusion of isotropy}.  (See Proposition \ref{prop:compactisomorphism} for a more direct approach in the case that $N$ is free.)
For any cone $\sigma$ in $\Sigma$, let $N_\sigma$ be the subgroup generated by $\{\beta(e_i) \mid i \in I_\sigma\}$.  We define a map
\begin{equation} \label{concrete inclusion of isotropy}
\gamma_\sigma:\Tor(N/N_\sigma) \to \Hom(\DG(\beta),\T)
\end{equation}
that depends on a choice of resolution
\begin{equation} \label{R-resolution}
0\longrightarrow {\Z^\ell} \stackrel{R}{\longrightarrow}  {\Z^{d-I+\ell}} \stackrel{q}{\longrightarrow}  N/N_\sigma \longrightarrow  0,
\end{equation}
and a lift $\tilde{q}:\Z^{d-I+\ell} \to N$ (cf.  \cite[Lemma 2.2.8]{Weibel94}), where $I=|I_\sigma|$.  Given $x\in \Tor(N/N_\sigma)$, choose a representative $a$ in $\Z^{d-I+\ell}$.  We shall define a homomorphism $\phi_a(\Z^{n+\ell})^*\rightarrow \C$ that descends to a homomorphism $\overline{\phi}_a:\DG(\beta) \rightarrow \C/\Z\cong \T$, and set $\gamma_\sigma(x)=\overline{\phi}_a$.

Set $J=n-I$, and write elements of $(\Z^{n+\ell})^\dual\cong (\Z^I)^\dual\oplus (\Z^J)^\dual\oplus(\Z^\ell)^\dual$ as triples $(u_I,u_J,v)$.  Define $\phi_a(u_I,u_J,v) = u_I(b) + v(c)$, where $b$ is the unique element in $\Z^I\otimes \C$ satisfying $(\beta_\sigma)_\C(b) =\tilde{q}_\C(a)$ and   $c$ is the unique element in $\Z^\ell\otimes \C$ satisfying $R_\C(c)=a$.  

To verify that $\phi_a$ descends to a homomorphism in $\Hom(\DG(\beta),\T)$, we choose the resolution 
$$
0 \longrightarrow \Z^\ell \stackrel{Q}{\longrightarrow} \Z^{d+\ell} \longrightarrow N \longrightarrow 0
$$
 with $Qv=(-\beta_\sigma^{-1} \tilde{q} R v,R v )$, and a lift $B:\Z^n \to {\Z^I \oplus \Z^{d-I+\ell}}$ satisfying $B(b,0)=a-Q(c)$, where $(b,0) \in  \Z^n\cong \Z^I\oplus \Z^{J}$.  (Such a lift $B$ can be obtained by first choosing a lift $B_J:\Z^J \to \Z^{d-I+\ell}$ and setting $B(z_I,z_J) =  (z_I + \beta_\sigma^{-1}(\beta(0,z_J)-\tilde{q} B_Jz_J),B_J z_J)$.) To see that  $\phi_a \circ [B\, Q]^\dual$ has image in $\Z \subset \C$, note that for $\begin{bmatrix} u_I & w & v \end{bmatrix}$ in $(\Z^{d+\ell})^\dual\cong(\Z^I)^\dual\oplus (\Z^{d-I})^\dual \oplus (\Z^{\ell})^\dual$,
\begin{align*}
\phi_a [B\, Q]^\dual \begin{bmatrix} u_I & w & v \end{bmatrix} 
&= \begin{bmatrix} u_I & w & v \end{bmatrix}[B\, Q] \begin{bmatrix}
b \\ 0 \\ c
\end{bmatrix} \\
&=  \begin{bmatrix} u_I & w & v \end{bmatrix} \left( B(b,0) + Q(c) 
 \right) \\
 &=v(a) \in \Z.
\end{align*}
 Notice that a different choice of representative $a'=a+R  w$ ($w\in \Z^\ell$) for $x$ leads to a homomorphism $\phi_{a'}$ that differs from $\phi_a$ by an integer-valued function $\phi_{R w}$; therefore, $\gamma_\sigma$ in  \eqref{concrete inclusion of isotropy} is well-defined.

\begin{proposition} \label{prop:concrete isotropy}
 The homomorphism $\gamma_\sigma$ in \eqref{concrete inclusion of isotropy} induces an isomorphism $\Tor(N/N_\sigma) \cong \Gamma_\sigma$.
\end{proposition}
\begin{proof}
We verify that $\gamma_\sigma$ induces an isomorphism $\Tor(N/N_\sigma) \cong \Gamma_\sigma$, by checking that $\gamma_\sigma$ is injective and that the composition 
$$
\Tor(N/N_\sigma) \stackrel{\gamma_\sigma}{\longrightarrow} G \to {\T}^n \to {\T}^J
$$
is trivial. If $\gamma_\sigma(x) = 0$, then $\phi_a$ is integer-valued, whence the corresponding element $c\in \C^\ell$ must actually lie in $\Z^\ell$ and thus $x={q}(a)=0$.  Therefore, $\gamma_\sigma$ is injective. Finally, the last two maps in the above composition are obtained by pulling back a homomorphism $\DG(\beta) \to {\T}$ along the quotient map $(\Z^{I+J+\ell})^\dual \to \DG(\beta)$ and the composite $(\Z^J)^\dual \to (\Z^{I+J})^\dual \to (\Z^{I+J+\ell})^\dual$ of inclusions.  By definition of $\gamma_\sigma$, this pullback is trivial, since $\phi_a$ is trivial on elements of the form $\begin{bmatrix}
0 & u_J & 0
\end{bmatrix}
$.
\end{proof}

\begin{remark} \label{remark:SNF}
In practice, the isotropy groups in Theorem \ref{theorem:isotropy} may be computed using the Smith Normal Form of a matrix, as we outline next.  As in the proof of the theorem, the isotropy group $\Gamma_\sigma$ is isomorphic to the torsion subgroup of the cokernel of the composition $\Z^{|I_\sigma|} \to \Z^n \stackrel{\beta}{\longrightarrow} N$. As in the discussion following Definition \ref{definition:stacky fan}, choose a free resolution 
$$
0 \to \Z^\ell \stackrel{Q}{\longrightarrow} \Z^{d+\ell} \to N
$$ 
and a lift $B:\Z^n \to \Z^{d+\ell}$ of $\beta$, and let $B_\sigma$ denote the restriction of $B$ to $\Z^{|I_\sigma|}$.  Then the commutative diagram
$$
\xymatrix{
0 \ar[r] & {\Z^\ell} \ar[r] \ar@{=}[d] & {\Z^{|I_\sigma|+\ell}} \ar[r] \ar[d]_{[{B_\sigma \, Q]}} & {\Z^{|I_\sigma|}} \ar[d]_{{\beta_\sigma}} \ar[r] & 0 \\
0 \ar[r] & {\Z^\ell} \ar[r]^{Q} & {\Z^{d+\ell}} \ar[r] & N \ar[r] & 0
}
$$
of short exact sequences shows that $\coker [B_\sigma \, Q]$ and $\coker \beta_\sigma$ are isomorphic.  Thus it suffices to compute the torsion submodule of the cokernel of the matrix $[B_\sigma \, Q]$.  The Smith Normal Form of $[B_\sigma \, Q]$ will be a $(d+\ell)\times (|I_\sigma| +\ell)$ matrix with non-zero entries $a_1, a_2, \ldots, a_{\mathrm{min}(d+\ell, |I_\sigma|+\ell)}$ appearing on the diagonal, satisfying the divisibility relations $a_j | a_{j+1}$.  The entries $a_j\neq 1$ give the orders of the cyclic subgroups appearing in the invariant factor decomposition of $\Gamma_\sigma$.
\end{remark}

 \begin{example} \label{eg:hbwithtorsionisotropycalc}
 To illustrate Remark \ref{remark:SNF}, we consider the following example.
 Let $N=\Z^2 \oplus \Z/2\Z$, and let $\Sigma$ be the fan in Figure \ref{fig:hb}, with ray generators $(1,0)$, $(0,1)$, $(0,-1)$ and $(-1,-2)$.  Set $\beta:\Z^4 \to N$ to be
 $$
 \beta(x,y,z,w)=(-2x+3z,-4x+6y-2w,x+y+z+w \operatorname{mod} 2).
 $$ 
\setlength{\unitlength}{2cm}
\begin{figure}
\centering
\begin{minipage}{.4\textwidth}
\centering
\begin{picture}(2.125,2.125)(0,-.7)
\put(-.2,.3){$\Sigma=$}
\thicklines
\put(1,.2){\vector(1,0){0.5}} 
\put(1,.2){\vector(0,1){.5}}
\put(1,.2){\vector(-1,-2){.5}} \put(.3,-.7){$\rho_1$}
\put(1,.2){\vector(0,-1){.5}}
\thinlines
\multiput(1.55,.2)(.1,0){4}{\line(1,0){.05}} \put(1.1,1){$\rho_2$}
\multiput(1,-.4)(0,-.1){5}{\line(0,1){.05}} \put(1.1,-.7){$\rho_4$}
\multiput(1,.75)(0,.1){4}{\line(0,1){.05}} \put(1.8,.3){$\rho_3$}
\put(1.25,.25){\line(-1,1){.2}}
\put(1.5,.25){\line(-1,1){.45}}
\put(1.75,.25){\line(-1,1){.7}}
\put(1.25, .15){\line(-1,-1){.2}}
\put(1.5,.15){\line(-1,-1){.45}}
\put(1.75,.15){\line(-1,-1){.7}}
\put(.95,-.05){\line(-1,0){.05}}
\put(.95,-.3){\line(-1,0){.15}}
\put(.95,-.55){\line(-1,0){.25}}
\qbezier(.825,-.05)(.8,0.25)(.95,.45)
\qbezier(.7,-.3)(.5,0.35)(.95,.7)
\qbezier(.575,-.55)(.2,.45)(.95,.95)
\end{picture} 
\end{minipage}
\begin{minipage}{.4\textwidth}
\centering
\begin{picture}(2.125,2.125)(0,-.7)
\put(-0.2,.3){$\Delta=$}
\thicklines
\put(.6,0){\line(1,0){1.8}}
\put(.6,0){\line(0,1){.6}}
\put(.6,.6){\line(1,0){.6}}
\put(1.2,.6){\line(2,-1){1.2}}
\end{picture} 
\end{minipage}
\caption{A polytope $\Delta$ in $\R^2\cong(N\otimes \R)^\dual$ and its dual fan $\Sigma=\Sigma(\Delta)$, from Example~\ref{eg:hbwithtorsionisotropycalc}.}
 \label{fig:hb}
\end{figure}
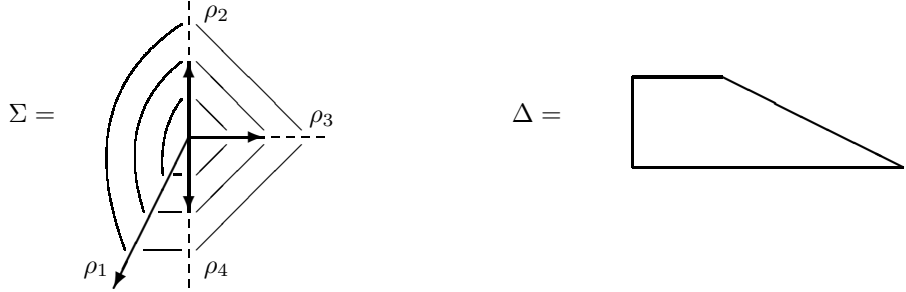
Fix the resolution $0 \to \Z \stackrel{Q}{\longrightarrow}  \Z^3 \to N \to 0$ with $Q= \begin{bmatrix}0 & 0 & 2
\end{bmatrix}^T$, and choose $B:\Z^4 \to \Z^3$ so that 
$$
[B\, Q] = \begin{bmatrix}
-2 & 0 & 3 & 0 & 0 \\
-4 & 6 & 0  &-2 & 0 \\
1 & 1 & 1 & 1 & 2
\end{bmatrix}.
$$
Let $\sigma$ be the cone in $\Sigma$ generated by $\rho_1$ and $\rho_2$.  Then since $$
[B_\sigma\, Q] = \begin{bmatrix}
-2 & 0 & 0 \\
-4 & 6  & 0 \\
1 & 1  & 2
\end{bmatrix} \text{ has Smith Normal Form }
\begin{bmatrix}
1&0&0\\
0&2&0\\
0&0&12
\end{bmatrix},
$$
we see that the isotropy group $\Gamma_\sigma \cong \mu_2 \times \mu_{12}$.
\end{example}

Notice that if $\sigma$ is maximal,  then $N_\sigma$ has the same rank as $N$. The following Corollary is immediate.  (Compare with \cite[Prop 4.3]{BCS05}.)

\begin{corollary}
Let $(N,\Sigma,\beta)$ be a stacky  fan, and $[Z_\Sigma/G]$ its corresponding toric DM stack.  If $z=(z_1, \ldots, z_n) \in Z_\Sigma$ has $d=\rank N$ vanishing coordinates, then $\sigma=\sigma_z$ is maximal and $\stab(z)\cong N/N_z$.
\end{corollary}

\subsection{Global quotients among toric DM stacks} \label{sec:globquot}

It is of interest to determine when a given stack $\mathcal{X}$ is a global quotient in the sense of stacks. 
The goal of this subsection is to characterize the global quotients among toric DM stacks in terms of the combinatorics of the stacky fan, and to give an explicit description of the quotient stack $[Z_\Sigma/G_0]$, which \cite{HaradaKrepski:2011} shows is the universal cover (in the sense of stacks) of a toric DM stack $[Z_\Sigma/G]$. 
Similar results are also obtained in \cite[Proposition 5.5, Corollary 5.7]{HaradaKrepski:2011}.

Let $G_0\subset G$ denote the connected component of the identity element.  By construction, the short exact sequence $0 \to \Tor(\DG(\beta)) \to \DG(\beta) \to \DG(\beta)/\Tor(\DG(\beta)) \to 0$ dualizes to give the short exact sequence $1 \to G_0 \to G \to G/G_0 \to 1$.  Below, we identify $G_0$ and the quotient $G/G_0$ in terms of stacky fan data.

\begin{lemma}\label{le:GmodG0}
Let $(\Sigma,N,\beta)$ be a stacky fan, and $G=\Hom(\DG(\beta),\T)$.
 If  $G_0 \subset G$ denotes the connected component of the identity element, then $G/G_0 \cong \coker \beta$.  In particular, $G$ is connected if and only if $\beta$ is surjective.
\end{lemma}
\begin{proof}
As $G=\Hom(\DG(\beta),\T)$, the group of connected components $G/G_0$ is $\Hom(\Tor(\DG(\beta)),\T)$, where $\Tor(\DG(\beta))$ denotes torsion submodule of $\DG(\beta)$. This torsion submodule may be identified by computing  $\Ext(\DG(\beta),\Z)$, which we show next is isomorphic to $\coker \beta$.

Note that $\coker \beta \cong \coker [B \, Q]$, where $[B\, Q]: \Z^{n+\ell} \to \Z^{d+\ell}$ is the homomorphism described in Section~\ref{se:prelim}.  This can be seen by verifying that the surjective composition $\Z^{d+\ell} \to N \to \coker \beta$ has kernel $\im [B \,Q]$.  Therefore, it suffices to show that $\Ext(\DG(\beta),\Z) \cong \coker [B\, Q]$.

To that end,  consider the free resolution that defines $\DG(\beta)$:
$$
\xymatrix{
0 \ar[r] & (\Z^{d+\ell})^\dual \ar[r]^{[B \, Q]^\dual} & (\Z^{n+\ell})^\dual \ar[r] & \DG(\beta) \ar[r] & 0
}
$$
(Note that since $\coker [B\, Q] \cong \coker \beta$ is assumed to be finite, then $[B\, Q]^\dual$ is injective.)  Applying $\Hom(-,\Z)$ and taking homology shows that $\Ext(\DG(\beta),\Z) \cong \coker [B\, Q]$, as required.
\end{proof}

\begin{remark} By \cite[Cor. 3.5]{HaradaKrepski:2011}, Lemma \ref{le:GmodG0} shows that the (stacky) fundamental group of $\mathcal{X}(N,\Sigma,\beta)$ is therefore isomorphic to $\coker \beta$.  (Cf. \cite[Section 3.2]{Fulton93})
\end{remark}

\begin{remark} As in the proof of the preceding Lemma, $\coker \beta \cong \coker [B\,Q]$ and thus the invariant factor decomposition of $G/G_0$ may be determined immediately from the Smith Normal Form of the matrix $[B\, Q]$.  For example, the reader may verify that $G/G_0 \cong \Z/2\Z$ in Example \ref{eg:hbwithtorsionisotropycalc}.
\end{remark}

Given a stacky fan $(N,\Sigma,\beta)$, we may model the quotient stack $[Z_\Sigma/G_0]$ as a toric DM stack of a related stacky fan $(N_0, \Sigma_0,\beta_0)$, defined as follows.  Consider the submodule $N_0 = \im(\beta) \subset N$, and let $\beta_0: \Z^n\rightarrow N_0$ be given by $\beta$ with its restricted codomain. Finally, we let $\Sigma_0$ be the fan in $N_0\otimes \R$ corresponding to $\Sigma$ defined by the natural isomorphism $N_0\otimes \R \cong N\otimes \R$ induced by the inclusion $N_0\subset N$ (where we have used the fact that $\beta$ has finite cokernel).

\begin{proposition} \label{prop:stackyfanG0}
 Let $(N, \Sigma, \beta)$ be a stacky fan and $\mathcal{X}(N,\Sigma,\beta)=[Z_\Sigma/G]$ its corresponding toric DM stack. The triple  $(N_0, \Sigma_0, \beta_0)$  defined as above is a stacky fan whose corresponding toric DM stack is $\mathcal{X}(N_0,\Sigma_0,\beta_0)=[Z_\Sigma/G_0]$. 
\end{proposition}

\begin{proof} It is straightforward to verify that $(N_0, \Sigma_0, \beta_0)$ defines a stacky fan.  Since $\Sigma_0$ and $\Sigma$ contain the same combinatorial information,  $J(\Sigma)=J(\Sigma_0)$, and thus $Z_{\Sigma_0}=Z_\Sigma$. 

 It remains to verify that the group action on $Z_\Sigma$ determined by the stacky fan $(N_0,\Sigma_0,\beta_0)$ is the same as that obtained by the restriction of the action of $G$ to the connected component of the identity $G_0$ on $Z_\Sigma$.

 To see this, we apply Lemma 2.3 of \cite{BCS05} to the following diagram of short exact sequences.
$$
\xymatrix@C=1.5em@R=1.5em{
 0 \ar[r] & \Z^n \ar@{=}[r] \ar[d]_{\beta_0} & \Z^n \ar[r] \ar[d]^{\beta} & 0 \ar[r] \ar[d] & 0 \\
0\ar[r] & {N_0} \ar[r] & N \ar[r] & \coker \beta \ar[r] & 0\\
 }
$$
Since $\DG(\{0\} \to \coker(\beta))$ can be naturally identified with $\coker(\beta)$, we obtain the diagram below with exact rows.
$$
\xymatrix@C=1.5em@R=1.5em{
 0 \ar[r] & 0 \ar[r] \ar[d] & (\Z^n)^\star \ar@{=}[r] \ar[d]_{\beta^\vee} & (\Z^n)^\star \ar[r] \ar[d]^{(\beta_0)^\vee} & 0 \\
0\ar[r] & \coker \beta \ar[r] & \DG(\beta) \ar[r] & \DG( \beta_0) \ar[r] & 0\\
 }
$$
This shows that $\DG(\beta_0)$ and $\DG(\beta)$ have the same rank and thus 
$\Hom(\DG(\beta_0),\T)$ 
 and $G$ have the same dimension.  To show that $\Hom(\DG(\beta_0),\T)$  is connected, it suffices to verify that $\DG(\beta_0)$ is torsion free, which follows from Lemma~\ref{le:GmodG0}.

Thus $\Hom(\DG(\beta_0),\T)=G_0$, the connected component of the identity in $G$. Lastly, note that the $G_0$ action on $\Z_\Sigma$ is induced by the composition $(\Z^n)^\dual \stackrel{\beta^\vee}{\longrightarrow} \DG(\beta)\longrightarrow \DG(\beta_0)$ so that $G_0$ acts via its inclusion into $G$, as desired.
\end{proof}

\bigskip

The above discussion yields a combinatorial condition characterizing global quotients among toric DM stacks. 
\begin{corollary} \label{cor:globalquotientcondition}
 Let $(N,\Sigma,\beta)$ be a stacky fan, and let $N_0=\im(\beta)$.
\begin{enumerate}
\item $N_\sigma=N_0$ for all maximal cones $\sigma\in\Sigma$ if and only if $G_0$ acts freely on $Z_\Sigma$.
\item $N_\sigma=N$ for all maximal cones $\sigma\in\Sigma$ if and only if $G$ is connected and acts freely on $Z_\Sigma$. 
\end{enumerate}
\end{corollary}

As mentioned above, the first item in Corollary \ref{cor:globalquotientcondition}  gives a criterion for detecting global quotients among toric DM stacks.  In particular, it shows that if $N_\sigma=N_0$ for all maximal cones, the toric DM stack $\mathcal{X}(N_0,\Sigma_0,\beta_0)=[Z_\Sigma/G_0]$ is in fact a smooth manifold.  As shown in \cite{HaradaKrepski:2011}, this exhibits the toric DM stack $\mathcal{X}(N,\Sigma,\beta)$ as a global quotient in this case---more precisely, when $N_\sigma=N_0$ for all maximal cones $\sigma \in \Sigma$, there is a natural equivalence of stacks $\mathcal{X}(N,\Sigma,\beta)\cong [(Z_\Sigma/G_0)/\Lambda]$ where  $\Lambda = G/G_0$.

\begin{example} Consider the toric DM stack $[Z_\Sigma/G]$ from Example \ref{eg:segment}, the labelled line segment with labels $r$ and $s$. Since $N_0 = g\Z$ where $g=\gcd(r,s)$, $[Z_\Sigma/G]$ is a global quotient if and only if $r=g=s$.  In that case, the $G\cong \T\times \mu_r$-action on $Z_\Sigma$ is induced by the homomorphism $G\to \T^2$, $(t,\xi_r^k)\mapsto (t,t\xi_r^k)$.  Therefore, $Z_\Sigma/G_0 =\P^1$, and the residual $\Lambda=G/G_0 \cong \mu_r$-action may be written in homogeneous coordinates as $\xi_r^k \cdot [z_0:z_1] = [z_0:\xi_r^k z_1]$, and $[Z_\Sigma/G] \cong [\P^1/\mu_r]$.
 \end{example}

\bigskip
\begin{remark} \label{remark:about component map}
Analogous to the map constructed in Proposition \ref{prop:concrete isotropy} modelling the inclusion of the isotropy groups, we may also model the quotient map $G \to G/G_0$ (cf. Lemma \ref{le:GmodG0}) more concretely as in the discussion that follows.  See Proposition \ref{eq:explicitGmodG0} for a more direct approach in the case that $N$ is free.
\end{remark}

Recall that the quotient $G \to G/G_0$ is the map $\Hom(\DG(\beta),{\T}) \to \Hom(\Tor(\DG(\beta)),{\T})$, induced by the inclusion of the torsion submodule  $\Tor(\DG(\beta)) \hookrightarrow \DG(\beta)$.  We shall describe an explicit isomorphism
$$
\Hom(\Tor(\DG(\beta)),{\T}) \stackrel{\cong}{\longrightarrow} \coker [B\, Q].
$$
Consider the diagram of short exact sequences,
$$
\xymatrix{
0 \ar[r] & {(\Z^{d+\ell})^\dual} \ar[r] \ar@{=}[d] & P \ar[r] \ar@{^{(}->}[d] & \Tor(\DG(\beta)) \ar[r] \ar@{^{(}->}[d] & 0 \\ 
0 \ar[r] & {(\Z^{d+\ell})^\dual} \ar[r]^{[B\, Q]^\dual}   & {(\Z^{n+\ell})^\dual} \ar[r] & {\DG(\beta)} \ar[r] & 0
}
$$
obtained by restriction (pullback) to $\Tor(\DG(\beta))$. Given a homomorphism $\varphi: \Tor(\DG(\beta)) \to \T=\C/\Z$, choose a homomorphism $\hat\varphi: P \to \C$ covering $\varphi$.  The restriction $\hat\varphi |_{(\Z^{d+\ell})^\dual}$ is integer-valued, and hence defines a vector $v_{\hat\varphi} \in \Z^{d+\ell}$ by duality (i.e. $\hat\varphi |_{(\Z^{d+\ell})^\dual}(u)=uv_{\hat\varphi}$ for all row vectors $u \in (\Z^{d+\ell})^\dual$). 

 Any two covers $\hat\varphi_1$, $\hat\varphi_2$ of $\varphi$ differ by a homomorphism $\alpha:P \to \Z$, which by restriction to $(\Z^{d+\ell})^\dual$ defines a vector $v_\alpha \in \Z^{d+\ell}$ that is the difference between $v_{\hat\varphi_1}$ and $v_{\hat\varphi_2}$.  We check that $v_\alpha$ is in the image of $[B\, Q]:\Z^{n+\ell} \to \Z^{d+\ell}$, and hence the correspondence $\varphi \mapsto v_{\hat\varphi}$ descends to a well-defined homomorphism 
$\Hom(\Tor(\DG(\beta)),{\T}) \to \coker [B\, Q]$.  Applying the Snake Lemma to the diagram of short exact sequences above shows that the quotient $(\Z^{n+\ell})^\dual/P \cong \DG(\beta)/\Tor(\DG(\beta))$ is free; therefore, $\alpha:P\to \Z$ may be extended to a homomorphism $\tilde\alpha: (\Z^{n+\ell})^\dual \to \Z$, which by duality defines a vector $w_{\tilde\alpha}$ in $\Z^{n+\ell}$.  Then for any row vector $u \in (\Z^{d+\ell})^\dual$, we have
$$
uv_{\alpha} = \alpha|_{ (\Z^{d+\ell})^\dual}(u) = \tilde\alpha(u [B\, Q]) = u [B\, Q]w_{\tilde\alpha},
$$
and hence $v_\alpha = [B\, Q]w_{\tilde\alpha}$, as required.

A similar argument shows that $\Hom(\Tor(\DG(\beta)),{\T}) \to \coker [B\, Q]$ is injective, which implies it is also surjective since these groups are known to be abstractly isomorphic finite groups.

\subsection{Isotropy and labelled polytopes} \label{sec:Nfree}
Let $N$ be a free $\Z$-module of rank $d$.
Viewing stacky polytopes as labelled polytopes, as in Section \ref{definition:labelled polytope}, we now give a more direct description of the maps modelling the inclusion of isotropy groups $\stab(z) \hookrightarrow G$ (see Theorem \ref{theorem:isotropy} and Proposition \ref{prop:concrete isotropy}), and the quotient map $G\to G/G_0$ (see Lemma \ref{le:GmodG0} and the discussion following Remark \ref{remark:about component map}). 

\medskip

Let $(\Delta, \{ m_i \}_{i=1}^{n})$ in $(N\otimes \R)^\dual \cong (\R^d)^\dual$ be a labelled polytope.  Let $\beta:\Z^n \to N$ be given by $\beta(e_i)=m_i\nu_i$, where $m_i\nu_i$ are the weighted normals to the facets of $\Delta$, and consider the resulting homomorphism $\bar\beta:(S^1)^n \to (S^1)^d$ with kernel $\KD=\ker \bar\beta$. Let $\exp: \Z^n \otimes \R \to (S^1)^n$ denote the exponential map.
 
 \medskip
 
 We begin with an explicit description of the stabilizers of Theorem~\ref{theorem:isotropy}. 

\begin{proposition}\label{prop:compactisomorphism}
Let $(\Delta, \{ m_i\}_{i=1}^n )$ be a labelled polytope in $(N\otimes \R)^\dual$ with primitive inward pointing facet normals $\nu_1 \otimes 1, \ldots, \nu_n \otimes 1$, and  $[\mu^{-1}(\tau)/K_D]$ its corresponding toric DM stack. For $z=(z_1, \ldots, z_n)$ in $\mu^{-1}(\tau) \subset \C^n$, let $N_z \subset N$ denote the submodule generated by $\{ m_i\nu_i \mid z_i=0 \}$.  Then the canonical map
$$
\Tor(N/N_z) \to K_D, \quad x+N_z \mapsto \exp(y\otimes \tfrac{1}{m}),
$$
is an isomorphism onto its image $\stab(z)$,
where $y$ is the unique element in $\operatorname{span}\{ e_i \mid z_i=0\} \subset \Z^n$  satisfying $\beta(y)=mx$ for some smallest positive integer $m$.
\end{proposition}

\begin{proof}
Let $\sigma$ be a cone in the normal fan $\Sigma(\Delta)$ with $I_\sigma=I_z$.
Analogous to Section \ref{sec:generalisotropy}, we see that $\Stab(z)$ is given by the kernel $\Gamma_\sigma$ of the composition (Cf. (\ref{eq:kernel})) 
$$
\KD \to (S^1)^n \to (S^1)^{|J_\sigma|}.
$$
We exhibit an isomorphism 
$
\psi:  \Tor(N/N_\sigma)\to \Gamma_\sigma ,
$
for any cone $\sigma$ of the normal fan $\Sigma(\Delta)$.

 Define $\psi:\Tor(N/N_\sigma) \to \Gamma_\sigma$ as follows.
Given $x + N_\sigma$ of order $m$, we may find a unique (see the Claim in Theorem 4.2) $y \in \Z^{I_\sigma}\subset \Z^{n}$ with $\beta(y) = mx$.  Consider $y\otimes \tfrac{1}{m} \in \Z^{n}\otimes \R$.  Since $\bar{\beta}(\exp(y\otimes \tfrac{1}{m})) =
\exp( x\otimes 1)=1$, $\exp(y\otimes \tfrac{1}{m}) \in \KD$.  And it is straightforward to see that ${q}_\sigma (\exp(y\otimes \tfrac{1}{m}) ) =1$, where ${q}_\sigma:(S^1)^{n} \to (S^1)^{|J_\sigma|}$ is the projection onto the ``non-trivial'' components; therefore, $\exp(y\otimes \tfrac{1}{m}) \in \Gamma_\sigma$ and we may set $\psi(x+N_\sigma) = \exp(y\otimes \tfrac{1}{m})$.

The map $\psi$ is well defined. Indeed, suppose we choose a different representative $x' + N_\sigma$ for $x+N_\sigma$, and let  $y'$ denote the corresponding element in $\Z^{I_\sigma}$ with $\beta(y') = mx'$.  Then there exists a (unique) $\eta \in \Z^{|I_\sigma|}$ satisfying $\beta(\eta) = x-x'$ and hence $\beta_\sigma(y-y') = \beta_\sigma(m\eta)$, which shows $y-y' = m\eta$.  Therefore, $\exp((y-y')\otimes \tfrac{1}{m}) = \exp (\eta \otimes 1) = 1$.

We check that $\psi$ is an isomorphism.  To check injectivity, suppose $\psi(x+N_{\sigma})=\exp(y\otimes \tfrac{1}{m})=1$, where $y$ and $m$ are as above.  Then $y\otimes \tfrac{1}{m}$ lies in the image of $\Z^{|I_{\sigma|}}\to \Z^{|I_{\sigma}|} \otimes \R$, which implies that $m=1$ so that $x \in N_\sigma$.  To check surjectivity, suppose $\gamma \in \Gamma_\sigma \subset (S^1)^{|I_\sigma|}$. Choose an element $v \in \Z^{|I_\sigma|} \otimes \R$ with $\exp(v) = \gamma$ and consider $(\beta_\sigma)_\R(v) \in N_\R$.  Since $\exp((\beta_\sigma)_\R(v))=1$, there is an element $x\in N$ such that $x\otimes 1 =\beta_\R(v)$.  We now check  that $\psi(x+N_\sigma) = \gamma$. Let $m$ be the order of $\gamma$.  Then $\exp(mv)=1$ and hence $mv = \xi \otimes 1$ for some $\xi$ in $\Z^{|I_{\sigma}|}$, and  $
\beta_\sigma (\xi) \otimes 1
= \beta_\R(mv) 
= mx\otimes 1
$
Therefore, $\beta_\sigma(\xi)=mx$, and $\exp(\xi \otimes \tfrac{1}{m})=\exp(v)=\gamma$, as required.
\end{proof}

\begin{remark}
The isomorphism given in Proposition~\ref{prop:compactisomorphism} is compatible with the one given in Proposition~\ref{prop:concrete isotropy} for general $\Z$-modules $N$.
Indeed,  if $N$ is assumed to be free, then we may choose an isomorphism $\tilde{q}:\Z^d \to N$, and take $R = \tilde{q}^{-1} \circ \beta_\sigma:\Z^I \to \Z^d$ in \eqref{R-resolution}.  With this choice, it can be shown that the following diagram commutes.
$$
\xymatrix@R-15pt{
& {K_D} \ar@{^{(}->}[r] \ar[dd] & { (S^1)^n \cong \Z^n \otimes S^1 }  \ar[dd] & {y\otimes \zeta} \ar@{|->}[dd]\\
\Tor(N/N_\sigma) \ar[ur] \ar[dr]_{\gamma_\sigma} & & & \\
& K \ar@{^{(}->}[r] & \Hom((\Z^n)^\dual, S^1) & {\nu \mapsto \zeta^{\nu(y)}}\\
}
$$
\end{remark}

Next, we give an explicit description of the component group $ \KD/(\KD)_0$, which models  Lemma \ref{le:GmodG0}, which shows that $\coker \beta \cong \KD/(\KD)_0$.

 Define a homomorphism $\varphi: N\to \KD/(\KD)_0$ as follows.  For $x\in N$, choose $y \in \Z^n \otimes \R$ with $\beta_\R(y)=x\otimes 1$ and consider $\exp(y) \in (S^1)^{n}$.  Then $\bar{\beta}(\exp(y)) =1$, and thus $\exp(y) \in \KD$.  Set $\varphi(x) = \exp(y)(\KD)_0$.  The map $\varphi$ is well-defined since $\ker \beta_\R$ is the Lie algebra of $\KD$, which exponentiates  onto $(\KD)_0$.  

The homomorphism $\varphi$ has kernel $\im(\beta)$.  Indeed, if  $\varphi(x)=\exp(y) \in (\KD)_0$, then there is an element $\zeta \in \ker \beta_\R$ with $\exp(y -\zeta) = 1$, and thus an integer vector $a \in \Z^n$ with  $ y-\zeta=a\otimes 1$.  Since $N\to N\otimes \R$ is injective ($N$ is free!), $x\otimes 1 =\beta_\R(y)=\beta_\R(a \otimes 1) = \beta(a)\otimes 1$ implies $\beta(a)=x$, and hence $\ker \varphi \subset \im(\beta)$.  Finally, if $x=\beta(a)$ for some $a\in \Z^n$, then we may choose $y=a\otimes 1$ to compute $\varphi(x)=\exp(y)(\KD)_0$ to see that $\exp(y)=1$, showing $\im \beta\subset \ker \varphi$.
Hence we obtain the following :

\begin{proposition} \label{eq:explicitGmodG0}
Suppose $N$ is torsion free.  The map $\varphi: N \to  \KD/(\KD)_0$ defined above descends to an isomorphism $\bar{\varphi}: \coker \beta \stackrel{\cong}{\longrightarrow} \KD/(\KD)_0$.
\end{proposition}

\begin{remark} The above arguments are essentially applications of the Snake Lemma (and parts of its proof).  For example, Proposition \ref{eq:explicitGmodG0} above follows part of the proof of the Snake Lemma for the diagram of short exact sequences,
$$
\xymatrix{
0 \ar[r] & {\ker \beta_\R} \ar[r] \ar[d] & {\Z^{n} \otimes \R} \ar[r]^{\beta_\R} \ar[d] & {N\otimes \R} \ar[r] \ar[d] & 0 \\
1 \ar[r] & \KD \ar[r] & {(S^1)^n} \ar[r] & {(S^1)^d} \ar[r] & 1
}
$$
and the map $\varphi$ is  the connecting homomorphism.
\end{remark}

\section{Weighted and Fake Weighted Projective Stacks}\label{se:WPS}

In this section we interpret the results in Section \ref{se:isotropy} for an important class of toric DM stacks known as weighted projective stacks.  In Section \ref{se:wps} we 
identify those toric DM stacks that are equivalent to weighted projective stacks in terms of their stacky fan data.  
Then in Section~\ref{se:fwps} we 
generalize our considerations to fake weighted projective stacks.

\subsection{Weighted projective stacks}\label{se:wps}
 As mentioned previously, weighted projective stacks are an important class of examples, and have been studied extensively both as stacks and as orbifolds. In this section we characterize those stacky polytopes corresponding to weighted projective stacks.

\begin{definition}
For positive integers $b_0, \ldots, b_d$, let $\T$ act on $\C^{d+1}
\smallsetminus \{0\}$   by $t\cdot(z_0, \ldots, z_d)~=~(t^{b_0}z_0, \ldots, t^{b_d}z_d)$. The resulting quotient stack $[(\C^{d+1}\smallsetminus \{0\})/\T]$ is called a \emph{weighted projective stack}, denoted $\calP(b_0, \ldots, b_d)$.   
\end{definition}

In particular, a weighted projective stack is a quotient by a {\em connected} one-dimensional Abelian Lie group action, but the action need not be effective.

Recall that to each $\calP(b_0, \ldots, b_d)$ there is associated stacky polytope $(N,\Delta, \beta)$ with $\DG(\beta)\cong \Z$ (see \cite[Example 21]{Sakai2010}) whose associated toric DM stack is equivalent to $\calP(b_0, \ldots, b_d)$.  Proposition \ref{prop:WPS} below shows that the toric DM stack corresponding to {\em any} stacky polytope satisfying $\DG(\beta)\cong \Z$ results in a weighted projective stack. 

\begin{proposition} \label{prop:WPS}
Let $(N, \Delta, \beta)$ be a stacky polytope, and let $\Sigma(\Delta)$ be the dual fan to $\Delta$. The associated toric DM stack $\mathcal{X}(N,\Sigma(\Delta),\beta)$ is a weighted projective stack $\calP(b_0, \ldots, b_d)$ 
if and only if 
$\DG(\beta) \cong \Z$.  In this case, the polytope $\Delta$ is a simplex, and the weights are determined by the condition that $(b_0,\dots, b_d)$ generates $\ker \beta \subset \Z^{d+1}$.
\end{proposition}

\begin{remark} \label{remark:cyclictorsion} {\rm 
Let $(N,\Delta, \beta)$ be a stacky polytope satisfying the condition $\DG(\beta)\cong \Z$. By Lemma \ref{le:GmodG0}, it follows that $\beta$ must be surjective.  Additionally, the torsion submodule $\Tor(N)$ of $N$ must be cyclic, and the proof of Proposition \ref{prop:WPS} shows that the order of $\Tor(N)$ is $g=\gcd(b_0, \ldots, b_d)$. }
\end{remark}

\begin{proof}[Proof of Proposition~\ref{prop:WPS}]
Since $\DG(\beta) \cong \Z$, the homomorphism $\beta$ of the stacky polytope $(N,\Delta,\beta)$ must have domain $\Z^{d+1}$ where (as usual) $\rank N=d$.  That is the polytope $\Delta$ has $d+1$ facets, and is therefore a simplex.
  Hence   $V(J(\Sigma(\Delta)))=\{0\}$ and $Z_{\Sigma(\Delta)} = \C^{d+1}\smallsetminus \{0\}$.  

We determine the $G$-action on $\C^{d+1}\smallsetminus \{0\}$.  Recall that the action is determined by applying $\Hom(-,\T)$ to the map $\beta^\vee: (\Z^{d+1})^\dual\rightarrow \DG(\beta)$, obtaining a homomorphism $G\to \T^{d+1}$.  We set out to determine $\beta^\vee$.

We first show that $\DG(\beta) \cong (\ker \beta)^\dual$. Let $0 \to \Z^\ell \stackrel{Q}{\longrightarrow} \Z^{d+\ell} \to N \to 0$ be a free resolution of $N$, and let $B:\Z^{d+1} \to \Z^{d+\ell}$ denote a lift of $\beta$. Then $\DG(\beta)=\coker[B\, Q]^*$. Similar to the proof of Proposition 2.2 in \cite{BCS05}, an application of the snake lemma to the diagram
$$
\xymatrix{
0 \ar[r] & \Z^\ell \ar@{=}[d] \ar[r] & \Z^{d+1+\ell}\ar[d]^{[B \, Q]} \ar[r] & \Z^{d+1} \ar[d]^{\beta} \ar[r] & 0 \\
0 \ar[r] & \Z^\ell \ar[r]^{Q} & \Z^{d+\ell} \ar[r] & N \ar[r] & 0
}
$$
shows that $[B\, Q]$ and $\beta$ have isomorphic cokernels, which are trivial by assumption, as well as isomorphic kernels.  It  follows that $\coker [B\, Q]^\dual \cong (\ker \beta)^\dual$.

Since $\beta: \Z^{d+1}\rightarrow N$ has finite cokernel, it follows there is an exact sequence
\begin{equation}\label{eq:exact}
\xymatrix{
0 \ar[r] &{(N^\dual)} \ar[r]^{\beta^\dual} & {(\Z^{d+1})^\dual} \ar[r]^{\beta^{\vee}} & {(\ker \beta)^\dual} \ar[r] 
& \Ext(N,\Z) \ar[r] & 0
}
\end{equation}
obtained by the identification $\DG(\beta)\cong (\ker\beta)^\dual$, where exactness at $\Ext(N,\Z)$ follows from the fact that $\Ext(N,\Z)$ is the first right derived functor of $\Hom(-,\Z)$ (and $N$ is finitely generated).

Since $\Delta$ is a simplex, we can find positive integers $b_0, \ldots, b_d$ such that $\sum_{j=0}^{d} b_j \beta(e_j)\otimes 1 =0$, where $\{e_0, \ldots e_d\}$ is the standard basis for $\Z^{d+1}$.  Without loss of generality, assume that $b=(b_0, \ldots, b_d)$ generates $\ker \beta$, which we now identify with $\Z$ according to $b \leftrightarrow 1$.
  The resulting identification $(\ker \beta)^\dual \cong \Z^\dual$ in ~\eqref{eq:exact} 
 gives $\beta^\vee$ the matrix representation $[b_0 \, \cdots \, b_d]$.  It follows that $\beta^\vee$ induces the homomorphism $\T\to \T^{d+1}$, $t\mapsto (t^{b_0}, \ldots, t^{b_d})$, which completes the proof.
\end{proof}

\subsection{Fake weighted projective stacks}\label{section:fWPS}\label{se:fwps}
Analogous to the fake weighted projective spaces,
in this section we consider a \emph{fake weighted projective stack}, a \emph{stack} quotient $\mathcal{W}/\Lambda$ where $\Lambda$ is a finite Abelian group acting (in the sense of group actions on stacks \cite{Romagny2005}, \cite{LermanMalkin2009}) on a weighted projective stack $\mathcal{W}=\mathcal{P}(b_0, \ldots, b_d)$. 
 Specifically, we characterize fake weighted projective stacks in terms of their associated combinatorial data.
 
By Proposition \ref{prop:WPS}, the toric DM stack associated to a stacky fan $(N,\Sigma,\beta)$ satisfying $\DG(\beta) \cong \Z$ is a weighted projective stack.  If we require only that $\rank \DG(\beta)=1$ (i.e. allowing $\DG(\beta)$  with torsion), the next Proposition shows that the resulting toric DM stack is a \emph{fake} weighted projective stack.  (Compare with \cite[Lemma 2.11]{Batyrev-Cox}.)  Note that the proof of the implication (\ref{fwps}) $\Rightarrow$ (\ref{rank})  of the proposition relies more heavily on the language of stacks, which in the interest of brevity will not be reviewed.  The reader may wish to consult the indicated references.

\begin{proposition} \label{prop:fakeWPS}
Let $(N, \Delta, \beta)$ be a stacky polytope, and let $\Sigma(\Delta)$ be the dual fan to $\Delta$. 
The following statements are equivalent:
\begin{enumerate}
\item the associated toric DM stack $\mathcal{X}(N,\Sigma(\Delta),\beta)$ is equivalent to a fake weighted projective stack; \label{fwps}
\item $\rank \DG(\beta) =1$ (i.e. $\dim G =1$); \label{rank}
\item $\Delta$ is a simplex. \label{simplex}
\end{enumerate}
Under any of the above conditions, $\mathcal{X}(N,\Sigma(\Delta),\beta)\cong \mathcal{P}(b_0, \ldots, b_d)/\Lambda$, where the weights are determined by the condition that $(b_0,\dots, b_d)$ generates $\ker \beta \subset \Z^{d+1}$.
\end{proposition}
\begin{proof} The equivalence of (\ref{rank}) and (\ref{simplex}) is immediate. 
Suppose $(N,\Delta,\beta)$ is a stacky polytope with $\rank \DG(\beta)=1$.
By \cite{LermanMalkin2009}, $\mathcal{X}(N,\Sigma(\Delta),\beta) =[Z_{\Sigma}/G] \cong [Z_{\Sigma}/G_0]/{\Lambda}$, where $\Lambda=G/G_0$ and $G_0$ denotes the connected component of the identity element in $G$.  By Proposition \ref{prop:stackyfanG0}, $[Z_{\Sigma}/G_0]$ is the toric DM stack associated to the stacky fan $(N_0,\Sigma(\Delta),\beta_0)$, which satisfies $\DG(\beta_0)\cong \Z$ and is thus a weighted projective stack, by Proposition \ref{prop:WPS}.

Conversely, suppose $\mathcal{X}(N,\Sigma(\Delta),\beta)$ is equivalent to a fake weighted projective stack $\mathcal{W}/\Lambda$ with $\mathcal{W}= \calP(b_0, \ldots, b_d)$.  We shall verify below that the quotient map of stacks $\mathcal{W} \to \mathcal{W}/\Lambda$ is a covering projection---i.e. a representable map of stacks such that every representative is a covering projection (see \cite{Noohi:2005}).  Since $\mathcal{W}$ has trivial (stacky) fundamental group, $\mathcal{W}$ is then the universal cover, which by \cite{HaradaKrepski:2011}, coincides with $\mathcal{X}(N_0,\Sigma_0,\beta_0)$.  Therefore $\rank \DG(\beta)=\rank \DG(\beta_0)=1$ by Proposition \ref{prop:WPS}.  The statement about the weights is also a direct consequence of Proposition \ref{prop:WPS}.

To see that $p:\mathcal{W} \to \mathcal{W}/\Lambda$ is a covering projection, it is enough to show that the base extension of $p$ along a presentation (also called chart or atlas) is a covering projection of topological spaces (since coverings are invariant under base change and local on the target---see \cite[Example 4.6]{Noohi:2005}).  Choose a $\Lambda$-atlas $X_1 \rightrightarrows X_0$ for $\mathcal{W}$ with $\Lambda$ acting freely on $X_1$ and $X_0$ (see \cite{LermanMalkin2009}).  Then $X_1/\Lambda \rightrightarrows X_0 /\Lambda$ is an atlas for the quotient $\mathcal{W}/\Lambda$.  It is straightforward to check that $X_0/\Lambda \times_{\mathcal{W}/\Lambda} \mathcal{W}\cong X_0$ (e.g. see the discussion of $2$-fiber products in \cite[Section 9]{Noohi:2005}), and hence the base extension of $p$ is $X_0 \to X_0/\Lambda$, which is a covering projection.
\end{proof}

\section{Labelled sheared simplices}
\label{se:shimplexWPS}

 We now introduce  \emph{labelled sheared simplices}, a natural family of stacky polytopes that includes many (but not all) stacky polytopes associated weighted and fake weighted projective stacks. We begin with the definition.

\begin{definition} \label{def:labelledshimplex}
Let  $\mathbf{a} = (a_1, \ldots, a_d) \in \Z^d$ be a primitive vector in the  positive orthant and $\{\epsilon_i\}$ the standard basis of $\R^d$. The \emph{sheared simplex} $\Delta(\mathbf{a})$ is  the convex hull  of the origin together with the points $\frac{\lcm(\mathbf{a})}{a_j}\epsilon_j$ $(j=1,\ldots,d)$ in $\R^d$.  Given positive integers $m_0, \ldots, m_d$, we define a \emph{labelled sheared simplex} as the stacky polytope $(\Z^d, \Delta(\mathbf{a}), \beta)$, where the homomorphism  $\beta:\Z^{d+1} \to \Z^d$ is given by $\beta(e_0)=-m_0 \mathbf{a}$ and $\beta(e_j)=m_j\epsilon_j$ ($j=1, \ldots, d$), where $e_0, \ldots, e_d$ denote the standard basis vectors for $\Z^{d+1}$.
\end{definition}

We first note that 
Proposition \ref{prop:fakeWPS} immediately implies that, for a labelled sheared simplex  $(\Z^d, \Delta(\mathbf{a}), \beta)$, the associated toric DM stack $[Z_{\Sigma(\Delta)}/G]$ is indeed a fake weighted projective stack.
The concrete combinatorial description of these simplices provide an interesting explicit class of examples 
of fake projective stacks to study.

Since all labelled sheared simplices give rise to fake weighted projective stacks, we
first identify, in Section~\ref{subsec:shimplexWPS}, precisely which labelled sheared simplices correspond to weighted projective stacks.
We analyze the fake weighted projective stacks arising from labelled sheared simplices more generally in Section~\ref{se:sheared}, and finally, we give a detailed analysis of 2-dimensional labelled sheared simplices in Section~\ref{se:2dimsheared}.

\subsection{Labelled sheared simplices corresponding to weighted projective stacks}\label{subsec:shimplexWPS}

The main result of this section is Proposition~\ref{prop:shimplexWPS}, in which we identify those labelled sheared simplices which 
give rise to a weighted projective stack.
By Proposition \ref{prop:WPS}, this involves translating the condition $\DG(\beta) \cong \Z$ in terms of the labels $\{ m_0, \ldots, m_d\}$ and the vector $\mathbf{a}$.   Since $\rank \DG(\beta)=1$,  $\DG(\beta) \cong \Z$ if and only if  $G$ is connected; therefore, we begin with the following lemma.

\begin{lemma}\label{prop:shimplexcomponentgroup} Let $(\Z^d,\Delta(\mathbf{a}),\beta)$  be a labelled sheared simplex with labels $\{m_0, \ldots, m_d\}$, and let $[Z_\Sigma/G]$ be its associated toric DM stack. If $G_0 \subset G$ denotes the connected component of the identity element, then 
$$
G/G_0 \cong \left[{\bigoplus_{i=0}^d \Z/m_i\Z }\right] \Big/ \langle (1\, \mathrm{mod}\, m_0, a_1\,\mathrm{mod}\, m_1, \ldots, a_d \,\mathrm{mod}\, m_d) \rangle
$$
\end{lemma}

\begin{example}
Consider the toric DM stack from Example \ref{eg:segment}, the labelled line segment with labels $r$ and $s$, where it was seen directly that $G \cong \T\times \mu_g$, where $g=\gcd(r,s)$ and $\mu_g \subset \T$ is the cyclic group of $g$-th roots of unity.  Hence $G/G_0 \cong \mu_g$.  This can be seen from the Lemma above as well, since
$
G/G_0 \cong (\Z/r\Z \oplus \Z/s\Z) / H
$
where $H$ is the subgroup generated by $(1,1)$.  
\end{example}

\begin{proof}[Proof of Lemma \ref{prop:shimplexcomponentgroup}]
By Lemma \ref{le:GmodG0}, $G/G_0 \cong \coker \beta$, which we now compute.  Consider the commutative diagram of short exact sequences
$$\xymatrixcolsep{5pc}
\xymatrix{
0 \ar[r] & \Z^{d+1} \ar[r]^{\times(m_0,\dots, m_d)} \ar[d]^{\beta} &  \Z^{d+1} \ar[r] \ar[d]^{\beta'} & {\bigoplus_{i=0}^d \Z/m_i\Z } \ar[r] \ar[d] & 0 \\
0 \ar[r] & \Z^d \ar@{=}[r] & \Z^d \ar[r] & 0 \ar[r] & 0
}
$$ where $\beta'$ may be written as the matrix $\begin{bmatrix}
-a_1 &1 & & 0 \\
\vdots & & \ddots & \\
-a_d & 0& & 1
\end{bmatrix}.$
Applying the Snake Lemma gives the   exact sequence
$$
0\to \ker \beta \to \ker \beta' \to  {\bigoplus_{i=0}^d \Z/m_i\Z } \to \coker \beta \to 0.
$$
Since $\ker \beta' \cong \Z$ is generated by $(1, a_1, \ldots, a_d) \in \Z^{d+1}$, the  second map in the sequence above sends the generator to $(1\, \mathrm{mod}\, m_0, a_1\,\mathrm{mod}\, m_1, \ldots, a_d \,\mathrm{mod}\, m_d)$. The result follows.
\end{proof}

\begin{proposition} \label{prop:shimplexWPS}
Let    $(\Z^d, \Delta(\textbf{a}), \beta)$ be a labelled sheared simplex with labels $\{m_0, \ldots, m_d\}$ and $\mathbf{a}=(a_1, \ldots, a_d)$. Let $M=m_0m_1 \cdots m_d$.
The following conditions are equivalent:
\begin{enumerate}
\item[(0)] $\DG(\beta)\cong \Z$;
\item[(1)] $\mathcal{X}(\Z^d,\Delta(\mathbf{a}),\beta)$ is equivalent to a weighted projective stack;
\item[(2)] $\displaystyle \gcd\left( \frac{M}{m_0}, \frac{Ma_1}{m_1}, \ldots, \frac{Ma_d}{m_d} \right)=1$;
\item[(3)] $\gcd(m_i, m_j) = 1$ for all $i \neq j$ in $\{0, \ldots, d\}$, and $\gcd(a_i, m_i) = 1$, for all $i$ in $\{1,\ldots,d\}$.
\end{enumerate}
In addition, if one of the above holds, then $\mathcal{X}(\Z^d,\Delta(\mathbf{a}),\beta) = \calP\left( \frac{M}{m_0}, \frac{Ma_1}{m_1}, \ldots, \frac{Ma_d}{m_d} \right)$.
\end{proposition}

\begin{proof}
The equivalence of conditions (0) and (1) is part of Proposition \ref{prop:WPS}, which also gives the identification of the weights.

Let $(\Z^d, \Delta(\textbf{a}), \beta)$ be a labelled sheared simplex, with labels $\{m_0, \ldots, m_d\}$ and $\mathbf{a}=(a_1,\ldots, a_d)$. Since the rank of $\DG(\beta)$  is 1, by Lemma~\ref{le:GmodG0}, condition (0) is equivalent to $\coker \beta=0$. 
By Lemma~\ref{prop:shimplexcomponentgroup}, $\coker \beta=0$ if and only if 
$$
H = \langle (1\, \mathrm{mod}\, m_0, a_1\,\mathrm{mod}\, m_1, \ldots, a_d \,\mathrm{mod}\, m_d) \rangle
$$
is equal to $\oplus_i \Z/m_i\Z$, or equivalently if  $H$ is cyclic of order $M=m_0m_1\cdots m_d$.  If (3) holds, this is immediate.

Conversely, suppose $H$ is cyclic of order $M$, and hence equals  $\oplus_i \Z/m_i\Z$. Therefore, $\gcd(m_i,m_j)=1$ for $i\neq j$. Consider the isomorphism $\oplus_i \Z/m_i\Z \to \Z/M\Z$ defined by sending $f_i \mapsto M/m_i$, where $f_i$ denotes the element whose only non-zero component is $1 \, \mathrm{mod} \, m_i$ in the $i$th component.  Under this isomorphism, the generator of $H$ is sent to $b=(M/m_0 + Ma_1/m_1 + \cdots + Ma_d/m_d) \, \mathrm{mod}\, M$, which has order $M/\gcd(M,b)=M$ since $H$ has order $M$.  Therefore, $\gcd(M,b)=1$ and hence $\gcd(m_j,a_j)=1$, which shows (3).

The equivalence of conditions (2) and (3) is straightforward to verify.
\end{proof}

Proposition \ref{prop:shimplexWPSisotropy} 
below describes the isotropy groups for weighted projective stacks that correspond to labelled sheared simplices.  For a weighted projective stack $\calP(b_0, \ldots, b_d)$, the resulting isotropy groups are straightforward to compute directly from the defining action of $\T$.   Namely, the isotropy of a point $z \in \C^{d+1} \smallsetminus\{0\}$  is easily seen to be cyclic of order $\gcd(b_j: z_j \neq 0)$.
 For those weighted projective stacks arising from labelled sheared simplices, we may use Proposition \ref{prop:shimplexWPS} to express this in terms of the labels $\{m_0, \ldots, m_d \}$ and $\mathbf{a}$.

\begin{proposition} \label{prop:shimplexWPSisotropy}
Suppose that $(\Z^d, \Delta(\textbf{a}), \beta)$ is a labelled sheared simplex
with labels $\{m_0, m_1, \ldots, m_{d}\}$ that corresponds to a
weighted projective stack. Let $z=(z_0, \ldots, z_d) \in \C^{d+1} \smallsetminus \{0\}$, and set $d_z=\mathrm{gcd}(a_j : z_j \neq 0)$, where $a_0$ is set to $1$.
 If $z_0\neq 0$,
 then  $\stab(z) \cong {\bigoplus_{i\in I_z} \Z/m_i \Z}$. If 
 $z_0=0$,
 then $\stab(z) \cong \Z/(m_z d_z)\Z$, where $m_z=\prod_{i\in I_z} m_i$.
\end{proposition}

The stacky polyope $(\Z^2,\Delta,\beta)$ introduced in
 Example~\ref{ex:GconnNfreeWPS} is a labelled sheared simplex $\Delta(1,1)$
 with labels $\{2,3,5\}$.  By Proposition~\ref{prop:shimplexWPSisotropy}, the
 associated toric DM stack is a weighted projective stack; therefore, we may
 use Proposition~\ref{prop:shimplexWPSisotropy} to reproduce the calculation
 in Example~\ref{eg:WPSisotropycalc}.

\begin{remark}
The results in this section depend only on the dual fan $\Sigma(\Delta)$, which consists of all cones  $\sigma$ generated by  any subset of the ray generators $\{{\bf a}, \epsilon_1,\ldots, \epsilon_d\}$.   
\end{remark}

\subsection{On fake weighted projective stacks arising from labelled sheared simplices}\label{se:sheared}

In this section we study labelled sheared simplices more generally. Specifically, we analyze the isotropy groups of the corresponding fake weighted projective stacks in Proposition~\ref{prop:shimplexisotropy}, and we characterize the labelled sheared simplices giving rise to global quotient stacks in Proposition~\ref{prop:shimplexglobalquotient}.

We begin with a result on isotropy groups.

\begin{proposition} \label{prop:shimplexisotropy} Let $(\Z^d,\Delta(\mathbf{a}),\beta)$  be a labelled sheared simplex with labels $\{m_0, \ldots, m_d\}$ and set $a_0=1$. Let $z=(z_0, \ldots, z_d) \in \C^d \smallsetminus \{0\}$, and set $d_z=\gcd (a_j :z_j\neq 0)$.
 The isotropy group $\stab(z)$
 is an extension
$$
0 \longrightarrow  {\bigoplus_{i\in I_z} \Z/m_i \Z} \longrightarrow {\stab(z)} \longrightarrow {\Z/d_z \Z} \longrightarrow 0. 
$$
\end{proposition}
\begin{proof}

Recall that by Theorem \ref{theorem:isotropy}, the isotropy group $\stab(z)$
is isomorphic to the torsion submodule of $N/N_z$.
 Form the matrix $B_z$ by deleting the $i$-th column of $\beta$ whenever $z_i\neq 0$
 Viewing $B_z$ as a homomorphism $\Z^{|I_z|} \to \Z^d$ realizes $N/N_z = \coker B_z$. We compute the torsion submodule of $\coker B_z$ next.

If $z_0 \neq 0$,
then one may readily see that $\Tor(N/N_z) \cong \bigoplus_{i\in I_z} \Z/m_i \Z$.  It remains to consider the case where $z_0=0$
(i.e. with corresponding matrix $B_z$ containing the first column of $\beta$).

To begin, observe that $\beta:\Z^{d+1} \to N=\Z^d$ factors as the composition $\Z^{d+1} \stackrel{L}{\longrightarrow} \Z^{d+1} \stackrel{\beta'}{\longrightarrow}\Z^d$, where
$$
L=
\begin{bmatrix}
m_0 & & 0 \\
& \ddots & \\
0& & m_d
\end{bmatrix}, \text{ and} \quad
\beta'=
\begin{bmatrix}
-a_1 &1 & & 0 \\
\vdots & & \ddots & \\
-a_d & 0& & 1
\end{bmatrix}.
$$
Accordingly, we may factor $B_z$ as a composition $\Z^{|I_z|} \stackrel{L_z}{\longrightarrow}  \Z^{|I_z|} \stackrel{B'_z}{\longrightarrow} \Z^d$, where $B'_z$ is a matrix whose first column is the first column of $\beta'$ and whose other columns are standard basis vectors $e_i \in \Z^d$ for $i\neq 0$  in $I_z$.  This factorization yields the following diagram of short exact sequences.
$$
\xymatrix{
0 \ar[r] & \Z^{|I_{z}|} \ar[r]^{L_{z}} \ar[d]^{B_{z}} &  \Z^{|I_{z}|} \ar[r] \ar[d]^{B'_{z}} & {\bigoplus_{i\in I_{{z}}} \Z/m_i\Z } \ar[r] \ar[d] & 0 \\
0 \ar[r] & \Z^d \ar@{=}[r] & \Z^d \ar[r] & 0 \ar[r] & 0
}
$$
Applying the Snake Lemma (and observing that $B'_{z}$ is injective) yields the exact sequence,
\begin{equation} \label{eq:snake}
\xymatrix{
0 \ar[r] & {\bigoplus_{i\in I_{z}} \Z/m_i \Z} \ar[r] & {\coker B_{z}} \ar[r] & {\coker B'_{z}} \ar[r] &0. \\
}
\end{equation}

Let $d_{z} = \mathrm{gcd} \{a_i : i \notin I_{z} \}$.  Then $B'_{z}$ is row equivalent to the matrix $C'_{z}$ obtained from $B'_{z}$ by replacing all $a_i$ with $i\notin I_{z}$ with 0's except for one which is replaced with $d_{z}$.  It follows that $\coker B'_{z} \cong \Z^{d-|I_{z}|} \oplus \Z/d_{z} \Z$.  Moreover, we may describe the generators of the free summand as follows.  
Let $E_{z}:\Z^d \to \Z^d$ denote the invertible homomorphism defined by $C'_{z} = E_{z} B'_{z}$.  Then  the free summand is generated by the images of the $E^{-1}(e_i)$ ($i\notin I_{z}$) in $\Z^d/\im B'_{z}$. It follows that the $E^{-1}(e_i)$ ($i\notin I_{z}$) must generate a free submodule in $\coker B_{z}$, which maps isomorphically onto the free summand of $\coker B'_{z}$.  Thus, we may pass to torsion submodules in (\ref{eq:snake}):
$$
0 \longrightarrow  {\bigoplus_{i\in I_{z}} \Z/m_i \Z} \longrightarrow {\Tor(\coker B_{z})} \longrightarrow {\Z/d_{z} \Z} \longrightarrow 0. 
$$
\end{proof}

As we will see in Section~\ref{se:2dimsheared}, the group extension in Proposition~\ref{prop:shimplexisotropy} can be non-trivial.  For example, consider the labelled sheared simplex $\Delta(1,2)$ with labels $m_0=2$, $m_1=4$, $m_2=1$.  By Proposition~\ref{corollary:shearedisotropy} in that Section (or by direct calculation), the isotropy for points of the form $(0,0,z) \in Z_{\Sigma(\Delta)}$ ($z\neq 0$) is $\Z/2\Z \oplus \Z/8\Z$, a non-trivial extension of $\Z/2\Z$ by $\Z/2\Z \oplus \Z/4\Z$.
\medskip

Using Corollary \ref{cor:globalquotientcondition}, we may also characterize those labelled sheared simplices yielding global quotient stacks.  Note that since the toric DM stack constructed from a labelled sheared simplex is a fake weighted projective stack, it is immediate that a global quotient in this case must then be a quotient of smooth projective stack $\calP^d$.

\begin{proposition} \label{prop:shimplexglobalquotient}
Let $(\Z^d,\Delta(\mathbf{a}),\beta)$  be a labelled sheared simplex with labels $\{m_0, \ldots, m_d\}$.
 The toric DM stack $\mathcal{X}(\Z^d,{\Delta(\mathbf{a})},\beta)$ is equivalent to a global quotient if and only if $m_i=m_0a_i$ for all $i=1\ldots,d$.
\end{proposition}

\begin{proof}
We apply Corollary \ref{cor:globalquotientcondition} and \cite[Theorem 4.4]{HaradaKrepski:2011}.   Let $\sigma_j$ be the maximal cone generated by the rays $\{\rho_0, \ldots, \widehat{\rho_j}, \ldots, \rho_d\}$, where $\widehat{\phantom{\rho_j}}$ signifies omission from the list.  Then 
$$
N'_{\sigma_0} = \mathrm{span} \{ m_1e_1, \ldots, m_de_d\}, \quad \text{and} \quad N'_{\sigma_j} = \mathrm{span} \{m_0 \sum_{i=1}^d a_i e_i, m_1 e_1, \ldots, \widehat{m_j e_j}, \ldots, m_d e_d \}.
$$
Observe that if $m_0a_i = m_i$ for all $i$, then it is clear that $N'=N'_{\sigma_j}$ for all $j=0\ldots, d$.  

To prove the converse, suppose $N'=N'_{\sigma_j}$ for all $j=0, \ldots d$.  For $j=0$, this implies that there exist $\alpha_1, \ldots, \alpha_d \in \Z$ such that
$$
\sum_{i=1}^d ( m_0 a_i- \alpha_i m_i )e_i =0,
$$
and thus $m_i | m_0 a_i$ for $i= 1\ldots d$.  Similarly, for $j=1, \ldots, d$, we see that there exist $\gamma_0, \ldots, \widehat{\gamma_j}, \ldots, \gamma_d \in \Z$ such that 
$$
m_j e_j = \gamma_0 m_0 \sum_{i=1}^d a_i e_i + \sum_{i\neq j} \gamma_i m_i e_i,
$$
or equivalently,
$$
\sum_{i\neq j} (\gamma_0 m_0 a_i + \gamma_i m_i) e_i + (\gamma_0 m_0 a_j -m_j)e_j =0.
$$
Therefore, $m_0a_j | m_j$ for $j=1, \ldots, d$ whence $m_0a_i = m_i$ for all $i$.  
\end{proof}

\subsection{$2$--dimensional labelled sheared simplices}\label{se:2dimsheared}

In this final section, we illustrate the results in previous sections by considering 
 the class of toric DM stacks arising from $2$-dimensional labelled sheared simplices.
As a consequence of Theorem \ref{theorem:isotropy}, we can now determine the isotropy groups corresponding to a labelled sheared simplex in the plane (see also Remark \ref{remark:SNF}).  

Let $\mathbf{a}=(a_1, a_2)$ be a primitive vector in the positive quadrant, and suppose $(\Z^2, \Delta(\mathbf{a}),\beta)$ is a labelled sheared simplex with labels $\{ m_0, m_1, m_2\}$.  Explicitly,  $\Delta(\mathbf{a})$ is the convex hull of the origin together with $(a_2,0)$ and $(0,a_1)$, with assigned labels $m_1$ to the edge along the $y$-axis, $m_2$ to the edge along the $x$-axis, and $m_0$ to the remaining edge (see Figure \ref{triangle}).  
\begin{figure}[h]
\setlength{\unitlength}{1.2cm}
\begin{center}
\begin{picture}(5,3.25)(-1.,-.5)
	\put(0,0){\circle*{.075}}
	\put(0,2){\circle*{.075}}
	\put(3,0){\circle*{.075}}
	\put(0,0){\line(0,1){2}}
	\put(0,0){\line(1,0){3}}
	\put(3,0){\line(-3,2){3}}
	\put(-.4,1){$m_1$}
	\put(1.4,-.25){$m_2$}
	\put(1.375,1.325){$m_0$}
	\put(-0.7,2.125){$\small{(0, a_1)}$}
	\put(3.125,-0.125){$\small{(a_2,0)}$}
	\put(-0.7,-0.125){$\small{(0,0)}$}
\end{picture}
\end{center}
\caption{A labelled sheared simplex in the plane.\label{triangle}}
\end{figure}
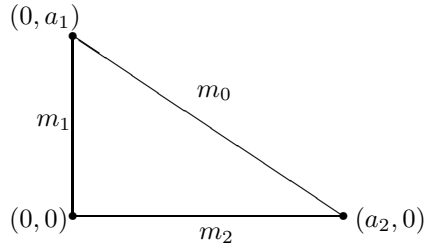

The isotropy for points  $(z_0, 0, 0) \in Z_\Sigma= \C^3\smallsetminus \{0\}$ with $z_0\neq0$ is  easily seen to be $\Z/m_1\Z \oplus \Z/m_2\Z$ (by Proposition \ref{prop:shimplexisotropy}). For points of the form $z=(0,0,z_2)$ with $z_2 \neq 0$---i.e. for points corresponding to the vertex $(0,a_1)$---we shall describe the isotropy $\stab(z)$ below.  (The isotropy for points of the form $z=(0,z_1, 0)$ with $z_1\neq 0$ can be obtained by exchanging the indices $1$ and $2$.)

We begin with a Lemma describing the Smith normal form of an integer matrix with exactly one zero entry.
\begin{lemma}\label{le:SNF}
 For non-zero  $a,b,c \in \Z$, the Smith Normal Form of $  \left[\begin{array}{cc}a & b \\ c & 0\end{array}\right] 
$ is
$ \left[\begin{array}{cc}g & 0 \\ 0 & b c/g\end{array}\right],
$ where
$g = \gcd(a,b,c)$.
\end{lemma}

\begin{proof}
Suppose that $\gcd(a,b) = d$.  Then, we claim that there exist $x,y \in \Z$ such that $x a  + y b = d$, and $\gcd(x,d)=1$.

To prove this claim, we first note that it is equivalent to the following:  Suppose $u,v \in \Z$ are relatively prime. Consider the set of solutions $X=\{ x \mid x u+y v=1 \mbox{ for some } y \in \Z \}$. For any given integer $d$,  there is some $x \in X$ so that $\gcd(x,d)=1$.

In this latter formulation, let $d \in \Z$ be given and suppose that $x_0$ is any solution to  $x_0 u +  y v =1$.  Recall that all solutions are then of the form $x = x_0 + t  v$ with $t\in \Z$.  Moreover, $x_0 u + y v = 1$ implies that $\gcd(x_0, v) =1$.  Then, showing that there is $x \in X$ such that $\gcd(x,d) =1$ is equivalent to showing that there is $t \in \Z$ such that $\gcd(x_0+t v, d)=1$.  We will construct such an integer, $t$.

Suppose that $d = p_1^{\alpha_1} \cdot \ldots \cdot p_s^{\alpha_s}$ is the prime factorization of $d$.  Let $t = \prod p_i$ such that $p_i$ does not appear in the prime factorization of either $x_0$ or $v$.  Because $x_0$ and $v$ are relatively prime, it follows that in the sum $x_0 + t v$ each prime in the factorization of $d$ appears exactly once.  That is, each $p_i$ divides exactly one of $x_0$ or $t v$.  Thus, $d$ cannot divide the sum and $\gcd(x,d) =1$.

Therefore, we may find $x$ and $y$ such that $a x + b y = d$ and $\gcd(x,d)=1$.  Moreover, since $\gcd(c x,d)  = \gcd(c,d) = \gcd(a, b, c) = g$,  there exist $p,q \in \Z$ such that $p (c x) + q d = g$. Hence,
$$
 \left[ \begin{array}{cc}  x & -b/d\\y & a/d\end{array}\right],\left[ \begin{array}{cc}  q & p\\ - c x/g& d/g \end{array}\right]  \in \mathrm{SL}_2(\Z),
 $$
  and  
 $$ 
 \left[ \begin{array}{cc}  q & p\\ - c x/g& d/g \end{array}\right] 
 \left[ \begin{array}{cc}  a  & b\\c & 0\end{array}\right]
 \left[ \begin{array}{cc}  x & -b/d\\y & a/d \end{array}\right] = 
 \left[ \begin{array}{cc}  g & -c b p/d\\0 & -b c /g\end{array}\right].
 $$
Recall that $d \mid b$ and that $g \mid c$, so $g \mid  (c b p/d)$.  Thus, by elementary column operations
\[  \left[ \begin{array}{cc}  g & -c b p/d\\0 & -b c /g\end{array}\right] \longleftrightarrow    \left[ \begin{array}{cc}  g & 0\\0 & -b c /g\end{array}\right].\]

Finally,  note that $g \mid (b c /g)$, so the above is the desired Smith normal form.
\end{proof}

We now can give the explicit form of the isotropy groups of a labelled sheared simplex in the plane.

\begin{proposition}\label{corollary:shearedisotropy}
 Let $(\Z^2, \Delta(\mathbf{a}), \beta)$ be a labelled sheared simplex with labels $\{m_0, m_1, m_2\}$ and $[Z_\Sigma/G]$ its corresponding toric DM stack.  
The isotropy of   $z=(0,0,z_2) \in Z_\Sigma$ with $z_2 \neq 0$ is
$\stab(z)\cong  \Z/g \Z \oplus \Z/((m_0 m_1 a_2) /g)\Z$, where $g = \gcd(m_0, m_1)$.
\end{proposition}

\begin{proof}
We consider the map $\beta: \Z^3 \rightarrow  \Z^2$ given by the matrix
\[ \beta = \left[ \begin{array}{ccc} - m_0   a_1 & m_1 & 0\\ -m_0 a_2 & 0 & m_2\end{array}\right].\]
By Lemma \ref{le:SNF}, the Smith normal form of $B_{z}=\left[ \begin{array}{cc} - m_0  a_1 & m_1 \\ -m_0 a_2 & 0 \end{array}\right]$ is
$\left[\begin{array}{cc}g & 0\\0 &  m_0 m_1 a_2 /g \end{array}\right]$
since $g=\gcd(m_0,m_1)=\gcd(m_0a_1,m_1,m_0a_2)$, which by Theorem \ref{theorem:isotropy} gives the result.
\end{proof}

Though Proposition \ref{corollary:shearedisotropy} gives the general form of the  isotropy group of points corresponding to the vertex $(0,a_1)$ of a sheared simplex, it can be instructive to consider several special cases to illustrate the interplay of the facet labels and the geometry of the sheared simplex---see Table \ref{table:isotropy}.

\begin{table}[h] 
\begin{tabular}{|c| c| c|}
\hline \textbf{Labels} & \textbf{Lengths}  & $\stab(z)$ \\
\hline
\multirow{2}{*}{
$m_0=m_1=m_2=1$} & $a_1=a_2=1$ &  $\{1\}$; i.e., smooth\\
\cline{2-3}
 & $a_1,a_2$ arbitrary & $\Z/a_1\Z$\\
\hline
\multirow{2}{*}{
$m_0,m_1,m_2$ arbitrary} & $a_1 = a_2 =1$ &  $\Z/m_0 \Z \oplus \Z / m_1\Z$\\
\cline{2-3}
  & $a_1,a_2$ arbitrary  & $\Z/g\Z \oplus \Z/(m_0 m_1 a_2/g)\Z$\\
\hline
\end{tabular}
\caption{The isotropy group $\stab(z)$ corresponding to points of the form $(0,0,z_2)$ with $z_2 \neq 0$ (i.e. corresponding to the vertex $(0,a_1)$ of $\Delta$) for a toric DM stack corresponding to a labelled sheared simplex $(\Z^2,\Delta(\mathbf{a}),\beta)$. Here, $g=\gcd(m_0,m_1)$.}
\label{table:isotropy}
\end{table}

\def\cprime{$'$}


\begin{thebibliography}{10}

\bibitem{Alper-guide}
J.~Alper.
\newblock A guide to the literature on algebraic stacks,
  \url{https://maths-people.anu.edu.au/~alperj/papers/stacks-guide.pdf}
\newblock available from the author's webpage.

\bibitem{Batyrev-Cox}
V.~Batyrev and D.~Cox.
\newblock On the Hodge structure of projective hypersurfaces in toric varieties,
\newblock {\em Duke Math. J.} 75, 293--338, 1994.


\bibitem{BCEFFGK-stacks}
K.~Behrend, B.~Conrad, D.~Edidin, B.~Fantechi, W.~Fulton, L.~G\"ottssche, and
  A.~Kresch.
\newblock Algebraic stacks,
  \url{http://www.math.uzh.ch/index.php?pr_vo_det&key1=1287&key2=580&no_cache=1}.
\newblock in progress.

\bibitem{BoissiereMannPerroni:2009b}
S.~Boissi{\`e}re, {\'E}.~Mann, and F.~Perroni.
\newblock A model for the orbifold {C}how ring of weighted projective spaces.
\newblock {\em Comm. Algebra}, 37(2):503--514, 2009.


\bibitem{BCS05}
L.~A. Borisov, L.~Chen, and G.~G. Smith.
\newblock The orbifold {C}how ring of toric {D}eligne-{M}umford stacks.
\newblock {\em J. Amer. Math. Soc.}, 18(1):193--215 (electronic), 2005.

\bibitem{Bu02}
W.~Buczy\'nska.
\newblock Fake weighted projective space.
\newblock 2002, arXiv:0805.1211 [math.AG].
\newblock Master's thesis Warsaw University.

\bibitem{Edidin:2003}
D.~Edidin.
\newblock What is a stack?
\newblock {\em Notices Amer. Math. Soc.}, 50(4):458‚Äì--459, 2003.

\bibitem{Fantechi:2001}
B.~Fantechi.
\newblock Stacks for everybody.
\newblock In {\em European {C}ongress of {M}athematics, {V}ol. {I}
  ({B}arcelona, 2000)}, volume 201 of {\em Progr. Math.}, pages 349--359.
  Birkh\"auser, Basel, 2001.

\bibitem{FantechiMannNironi:2010}
B.~Fantechi, E.~Mann, and F.~Nironi.
\newblock Smooth toric {D}eligne-{M}umford stacks.
\newblock {\em J. Reine Angew. Math.}, 648:201--244, 2010.

\bibitem{Fulton93}
W.~Fulton.
\newblock {\em Introduction to toric varieties}, volume 131 of {\em Annals of
  Mathematics Studies}.
\newblock Princeton University Press, Princeton, NJ, 1993.
\newblock The William H. Roever Lectures in Geometry.

\bibitem{GeraschenkoSatriano:2011a}
A.~Geraschenko and M.~Satriano.
\newblock Toric stacks {I}: The theory of stacky fans,  to appear in {\em Transactions of the Amer. Math. Soc.}, 
  arxiv:1107.1906 [math.AG].

\bibitem{GeraschenkoSatriano:2011b}
A.~Geraschenko and M.~Satriano.
\newblock Toric stacks {II}: Intrinsic characterization of toric stacks, to appear in {\em Transactions of the Amer. Math. Soc.}, arxiv:1107.1907 [math.AG].

\bibitem{GHK05}
R.~Goldin, T.~S. Holm, and A.~Knutson.
\newblock Orbifold cohomology of torus quotients.
\newblock {\em Duke Math. J.}, 139(1):89--139, 2007.

\bibitem{HaradaKrepski:2011}
M.~Harada and D.~Krepski.
\newblock Global quotients among toric {D}eligne-{M}umford stacks, to appear in {\em Osaka J. Math.},
  arXiv:1302.0385 [math.DG].

\bibitem{Iwanari:2009b}
I.~Iwanari.
\newblock Logarithmic geometry, minimal free resolutions and toric algebraic
  stacks.
\newblock {\em Publ. Res. Inst. Math. Sci.}, 45(4):1095--1140, 2009.

\bibitem{Jiang:2007}
Y.~Jiang.
\newblock The {C}hen-{R}uan cohomology of weighted projective spaces.
\newblock {\em Canad. J. Math.}, 59(5):981--1007, 2007.


\bibitem{Ka09}
A.~Kasprzyk.
\newblock Bounds on fake weighted projective spaces.
\newblock {\em Kodai Mathematical Journal}, 32:197--208, 2009.

\bibitem{Lerman:2010}
E.~Lerman.
\newblock Orbifolds as stacks?
\newblock {\em Enseign. Math. (2)}, 56(3-4):315--363, 2010.

\bibitem{LermanMalkin2009}
E.~Lerman and A.~Malkin.
\newblock Hamiltonian group actions on symplectic {D}eligne-{M}umford stacks
  and toric orbifolds.
\newblock {\em Advances in Mathematics}, 229(2):984--1000, January 2012.

\bibitem{LT97}
E.~Lerman and S.~Tolman.
\newblock Hamiltonian torus actions on symplectic orbifolds and toric
  varieties.
\newblock {\em Trans. Amer. Math. Soc.}, 349(10):4201--4230, 1997.

\bibitem{Mann:2008}
E.~Mann.
\newblock Orbifold quantum cohomology of weighted projective spaces.
\newblock {\em J. Algebraic Geom.}, 17(1):137--166, 2008.


\bibitem{Metzler03}
D.~Metzler.
\newblock Topological and smooth stacks, arXiv:math/0306176 [math.DG].

\bibitem{Noohi:2005}
B.~Noohi.
\newblock Foundations of topological stacks {I}, arXiv:math/0503247 [math.AG].

\bibitem{Perroni:2008}
F.~Perroni.
\newblock A note on toric {D}eligne-{M}umford stacks.
\newblock {\em Tohoku Math. J. (2)}, 60(3):441--458, 2008.

\bibitem{Romagny2005}
M.~Romagny.
\newblock Group actions on stacks and applications.
\newblock {\em Michigan Math. J.}, 53(1):209--236, 2005.

\bibitem{Sakai2010}
H.~Sakai.
\newblock The symplectic {D}eligne-{M}umford stack associated to a stacky
  polytope.
\newblock {\em Results Math.}, 63(3-4):903--922, 2013.

\bibitem{Weibel94}
C.~A. Weibel.
\newblock {\em An introduction to homological algebra}, volume~38 of {\em
  Cambridge Studies in Advanced Mathematics}.
\newblock Cambridge University Press, Cambridge, 1994.

\end{thebibliography}
\end{document}